\documentclass[12pt]{amsart}
\usepackage{amsfonts}

\usepackage{amscd}

\newtheorem{theorem}{Theorem}[section]
\newtheorem{corollary}[theorem]{Corollary}
\newtheorem{lemma}[theorem]{Lemma}
\newtheorem{proposition}[theorem]{Proposition}

\theoremstyle{definition}
\newtheorem{remark}{Remark}[section]
\newtheorem{definition}{Definition}[section]

\numberwithin{equation}{section}

\begin{document}
\title[Asymptotic equivalence]{Asymptotic Equivalence\\
for Nonparametric Regression}
\author[I. Grama and M. Nussbaum]{Ion Grama and Michael Nussbaum}
\address[Grama, I.]{\\
Universit\'e de Bretagne-Sud, Laboratoire SABRES \\
Rue Ives Mainguy\\
56000 Vannes, France}
\email{ion.grama@univ-ubs.fr}
\thanks{The first author was supported by the Alexander von Humboldt-Stiftung.
Research of the second author supported by the National Science Foundation
under Grant DMS 0072162}
\address[Nussbaum, M.]{\\
Department of Mathematics, Mallot Hall, Cornell University\\
Ithaca, NY 14853-4201, USA}
\email{nussbaum@math.cornell.edu}
\date{October, 2001. }
\dedicatory{\textit{Universit\'e de Bretagne-Sud and Cornell University} }
\maketitle

\begin{abstract}
We consider a nonparametric model $\mathcal{E}^{n},$ generated by
independent observations $X_{i},$ $i=1,...,n,$ with densities $p(x,\theta
_{i}),$ $i=1,...,n,$ the parameters of which $\theta _{i}=f(i/n)\in \Theta $
are driven by the values of an unknown function $f:[0,1]\rightarrow \Theta $
in a smoothness class. The main result of the paper is that, under
regularity assumptions, this model can be approximated, in the sense of the
Le Cam deficiency pseudodistance, by a nonparametric Gaussian shift model $%
Y_{i}=\Gamma (f(i/n))+\varepsilon _{i},$ where $\varepsilon
_{1},...,\varepsilon _{n}$ are i.i.d. standard normal r.v.'s, the function $%
\Gamma (\theta ):\Theta \rightarrow \mathrm{R}$ satisfies $\Gamma ^{\prime
}(\theta )=\sqrt{I(\theta )}$ and $I(\theta )$ is the Fisher information
corresponding to the density $p(x,\theta ).$
\end{abstract}

\section{Introduction\label{sec:Intro}}

Consider the nonparametric regression model with non-Gaussian noise
\begin{equation}
X_{i}=f\left( i/n\right) +\xi _{i},\quad i=1,...,n,  \label{LRG-1}
\end{equation}
where $\xi _{1},...,\xi _{n}$ are i.i.d. r.v.'s of means $0$ and finite
variances, with density $p(x)$ on the real line, and $f(t),$ $t\in \lbrack
0,1],$ is an unknown real valued function. It is well-known that, under some
assumptions, this model shares many desirable asymptotic properties of the
Gaussian nonparametric regression model
\begin{equation}
X_{i}=f\left( i/n\right) +\varepsilon _{i},\quad i=1,...,n,  \label{LRG-2}
\end{equation}
where $\varepsilon _{1},...,\varepsilon _{n}$ are i.i.d. standard normal
r.v.'s. In a formal way, two sequences of statistical experiments are said
to be \emph{asymptotically equivalent} if the Le Cam pseudo-distance between
them tends to $0$ as $n\rightarrow \infty .$ Such a relation between the
model (\ref{LRG-2}) and its continuous time analog was first established by
Brown and Low \cite{Br-Low}. In a paper by Nussbaum \cite{Nuss} the
accompanying Gaussian model for the density estimation from an i.i.d. sample
was found to be the white noise model (\ref{LRG-2}) with the root of the
density as a signal. The case of generalized linear models (i.e. a class of
nonparametric regression models with non-additive noise) was treated in
Grama and Nussbaum \cite{Gr-Nu-2}. However none of the above results covers
observations of the form (\ref{LRG-1}). It is the aim of the present paper
to develop an asymptotic equivalence theory for a more general class of
nonparametric models, in particular for the location type regression model (%
\ref{LRG-1}). In Section \ref{Sect-EX} we shall derive simple sufficient
conditions for the models (\ref{LRG-1}) and (\ref{LRG-2}) to be
asymptotically equivalent; a summary can be given as follows. Let $f$ be in
a H\"{o}lder ball with exponent $\beta >1/2.$ Let $p\left( x\right) $ be the
density of the noise variables $\xi _{i}$, and set $s\left( x\right) =\sqrt{%
p\left( x\right) }.$ Assume that the function $s^{\prime }\left( x\right) $
satisfies a H\"{o}lder condition with exponent $\alpha ,$ where $1/2\beta
<\alpha <1,$ and, for some $\delta >\frac{2\beta +1}{2\beta -1}$ and $%
\varepsilon >0,$%
\begin{equation*}
\sup_{\left| u\right| \leq \varepsilon }\int_{-\infty }^{\infty }\left|
\frac{s^{\prime }(x+u)}{s(x)}\right| ^{2\delta }p(x)dx<\infty .
\end{equation*}
Assume also that the Fisher information in the parametric location model $%
p\left( x-\theta \right) $, $\theta \in \mathrm{R}$ is $1.$ Then the models (%
\ref{LRG-1}) and (\ref{LRG-2}) are asymptotically equivalent. The above
conditions follow from the results of the paper for a larger class of
nonparametric regression models which we now introduce.

Let $p(x,\theta )$ be a parametric family of densities on the measurable
space $(X,\mathcal{X},\mu ),$ where $\mu $ is a $\sigma $-finite measure, $%
\theta \in \Theta $ and $\Theta $ is an interval (possibly infinite) in the
real line. Our nonparametrically driven model is such that at time moments $%
t_{1},...,t_{n}$ we observe a sequence of independent $X$-valued r.v.'s $%
X_{1},...,X_{n}$ with densities $p(x,f(t_{1})),...,p(x,f(t_{n})),$ where $%
f(t),$ $t\in \lbrack 0,1]$ is an unknown function and $t_{i}=i/n,$ $%
i=1,...,n.$ The principal result of the paper is that, under regularity
assumptions on the density $p(x,\theta ),$ this model is asymptotically
equivalent to a sequence of homoscedastic Gaussian shift models, in which we
observe
\begin{equation*}
Y_{i}=\Gamma (f(t_{i}))+\varepsilon _{i},\quad i=1,...,n,
\end{equation*}
where $\varepsilon _{1},...,\varepsilon _{n}$ are i.i.d. standard normal
r.v.'s, $\Gamma (\theta ):\Theta \rightarrow \mathrm{R}$ is a function such
that $\Gamma ^{\prime }(\theta )=\sqrt{I(\theta )}$ and $I(\theta )$ is the
Fisher information in the parametric family $p(x,\theta ),$ $\theta \in
\Theta .$

The function $\Gamma (\theta )$ can be related to the so called variance
stabilizing transformation (see Remark \ref{REMARK-VS_TR} below). In the
case of the location type regression model (\ref{LRG-1}), derived from the
family $p\left( x-\theta \right) $, $\theta \in \Theta $, we have $\Gamma
(\theta )=\theta .$ For other nontrivial examples we refer the reader to our
Section \ref{Sect-EX}, where it is assumed that the density $p(x,\theta )$
is in a fixed exponential family $\mathcal{E}$ (see also Grama and Nussbaum
\cite{Gr-Nu-2}). Notable among these is the binary regression model (cf.
\cite{Gr-Nu-2}): let $X_{i}$ be Bernoulli $0$-$1$ r.v.'s with unknown
probability of success $\theta _{i}=f(t_{i}),$ $i=1,...,n,$ where $f$ is in
a H\"{o}lder ball with exponent $\beta >1/2$ and is separated form $0$ and $%
1.$ Then the accompanying Gaussian model is
\begin{equation*}
Y_{i}=2\arcsin \sqrt{f(t_{i})}+\varepsilon _{i},\quad i=1,...,n.
\end{equation*}
The function $\Gamma (\theta )=2\arcsin \sqrt{\theta }$ is known to be the
variance-stabilizing transformation related to the Bernoulli random
variables.

The global result above is derived from the following local result. Define a
local experiment to be generated by independent observations $%
X_{1},...,X_{n} $ with densities $p(x,g(t_{1})),...,p(x,g(t_{n})),$ where $%
g(t)$ is in a certain neighborhood of ''nonparametric size'' (i.e. whose
radius is large relative to $n^{-1/2}$) of a fixed function $f.$ We show
that this model is asymptotically equivalent to a heteroscedastic Gaussian
model, in which we observe
\begin{equation}
Y_{i}=g(t_{i})+I(f(t_{i}))^{-1/2}\varepsilon _{i},\quad i=1,...,n.
\label{LOC-1}
\end{equation}
As an example, for binary regression (Bernoulli observations) with
probabilities of success $\theta _{i}=f(t_{i}),$ $i=1,...,n$ we obtain the
local Gaussian approximation (\ref{LOC-1}) with $I(\theta )=\frac{1}{\theta
(1-\theta )}.$

In turn, the local approximation (\ref{LOC-1}) is a consequence of a more
general local result for nonparametric models satisfying some regularity
assumptions, which is of independent interest. Namely, we show that if the $%
\log $-likelihood of a nonparametric experiment satisfies a certain
asymptotic expansion in terms of independent random variables, subject to a
Lindeberg type condition, and with a remainder term converging to $0$ at
some rate, then a version of this experiment can be constructed on the same
measurable space with a Gaussian experiment such that the Hellinger distance
between the corresponding probability measures converges to $0$ at some
rate. The proof of the last result is based on obtaining versions of the
likelihood processes on a common probability space by means of a functional
Hungarian construction for partial sums, established in Grama and Nussbaum
\cite{Gr-Nu-1}.

The abovementioned results are part of an effort to extend Le Cam's
asymptotic theory (see Le Cam \cite{LeCam} and Le Cam and Yang \cite
{LCY-2000}) to a class of general models with infinite dimensional
parameters which \emph{cannot be estimated} at the usual ''root-$n$'' rate $%
n^{-1/2}.$ The case of the infinite dimensional parameters which are
estimable at this rate was considered for instance in Millar \cite{Mill}.
The approach used in the proofs of the present results is quite different
from that in the ''root-$n$'' case and was suggested by the papers of Brown
and Low \cite{Br-Low} and Nussbaum \cite{Nuss} (see also Grama and Nussbaum
\cite{Gr-Nu-2}).  An overview of the technique of proof can be found at the
end of the Section \ref{sect Res}.

A nonparametric regression model with random design, but Gaussian noise was
recently treated by  Brown,  Low and Zhang \cite{Br-L-Zh}. We focus here on
the nongaussian case, assuming a regular nonrandom design $i/n$, $i=1,\ldots
,n$: the model is generated by a parametric family of densities $p(x,\theta
),$ $\theta \in \Theta $, where $\theta $ is assumed to take the values $%
f(i/n)$ of the regression function $f$. The term \textit{nonparametrically
driven parametric model }shall also be used for this setup.

The paper is organized as follows. Section \ref{SECTION-BACKGROUND} contains
some background on asymptotic equivalence. Our main results are presented in
Section \ref{sect Res}. In Section \ref{Sect-EX} we illustrate the scope of
our regularity assumptions by considering the case of the location type
regression model and the exponential family model (known also as the
generalized linear model). In Section \ref{Sect-gen-appr} we prove our basic
local result on asymptotic equivalence for a general class of nonparametric
experiments. Then in Section \ref{SECTION-Appl} we just apply this general
local result to the particular case when the nonparametrically driven
parametric model satisfies the regularity assumptions of the Section \ref
{sect Res}. In Sections \ref{SECTION-Local} and \ref{SECTION-Global} we
globalize the previous local result twofold: first over the time interval $%
[0,1]$, and then over the parameter set $\Sigma ^{\beta }.$ The global form
over $\Sigma ^{\beta }$ requires a reparametrization of the family $%
p(x,\theta )$ using $\Gamma (\theta ).$ At the end of Section \ref
{SECTION-Local}, we show that this allows a homoscedastic form of the
accompanying local Gaussian experiment, which can readily be globalized.
Finally in the Appendix we formulate the functional Koml\'{o}s-Major-Tusn%
\'{a}dy approximation proved in \cite{Gr-Nu-1} and some well-known auxiliary
statements.

\section{Background on asymptotic equivalence\label{SECTION-BACKGROUND}}

We follow Le Cam \cite{LeCam} and Le Cam and Yang \cite{LCY-2000}. Let $%
\mathcal{E}=(X,\mathcal{X},\{P_{\theta }:\theta \in \Theta \})$ and $%
\mathcal{G}=(Y,\mathcal{Y},\{Q_{\theta }:\theta \in \Theta \})$ be two
experiments with the same parameter set $\Theta .$ Assume that $(X,\mathcal{X%
})$ and $(Y,\mathcal{Y})$ are complete separable (Polish) metric spaces. The
deficiency of the experiment $\mathcal{E}$ with respect to $\mathcal{G}$ is
defined as
\begin{equation*}
\delta \left( \mathcal{E},\mathcal{G}\right) =\inf_{K}\sup_{\theta \in
\Theta }\left\| K\cdot P_{\theta }-Q_{\theta }\right\| _{\text{Var}},
\end{equation*}
where the infimum is taken over the set of all Markov kernels $K:\left( X,%
\mathcal{X}\right) \rightarrow (Y,\mathcal{Y})$ and $\left\| \cdot \right\|
_{\text{Var}}$ is the total variation norm for measures. Le Cam's distance
between $\mathcal{E}$ and $\mathcal{G}$ is defined to be
\begin{equation*}
\Delta \left( \mathcal{E},\mathcal{G}\right) =\max \{\delta \left( \mathcal{E%
},\mathcal{G}\right) ,\delta \left( \mathcal{G},\mathcal{E}\right) \}.
\end{equation*}
An equivalent definition of the Le Cam distance is obtained if we define the
one sided deficiency $\delta \left( \mathcal{E},\mathcal{G}\right) $ as
follows: let $(D,\mathcal{D})$ be a space of possible decisions. Denote by $%
\Pi (\mathcal{E})$ the set of randomized decision procedures in the
experiment $\mathcal{E},$ i.e. the set of Markov kernels $\pi (x,A):(X,%
\mathcal{X})\rightarrow (D,\mathcal{D}).$ Define $\mathcal{L}(D,\mathcal{D})$
to be the set of all loss functions $L(\theta ,z):\Theta \times D\rightarrow
\lbrack 0,\infty )$ satisfying $0\leq L(\theta ,z)\leq 1.$ The risk of the
procedure $\pi \in \Pi (\mathcal{E})$ for the loss function $L\in \mathcal{L}%
(D,\mathcal{D})$ and true value $\theta \in \Theta $ is set to be
\begin{equation*}
R(\mathcal{E},\pi ,L,\theta )=\int_{X}\int_{D}L(\theta ,z)\pi
(x,dz)P_{\theta }(dx).
\end{equation*}
Then the one-sided deficiency can be defined as

\begin{equation*}
\delta \left( \mathcal{E},\mathcal{G}\right) =\sup \sup_{L\in \mathcal{L}(D,%
\mathcal{D})}\inf_{\pi _{1}\in \Pi (\mathcal{E})}\sup_{\pi _{2}\in \Pi (%
\mathcal{G})}\sup_{\theta \in \Theta }\left| R(\mathcal{E},\pi _{1},L,\theta
)-R(\mathcal{E},\pi _{2},L,\theta )\right| .
\end{equation*}
where the first supremum is over all possible decision spaces $(D,\mathcal{D}%
)$.

Following Le Cam \cite{LeCam}, we introduce the next definition.

\begin{definition}
\label{Definition-AE}Two sequences of statistical experiments $\mathcal{E}%
^{n},$ $n=1,2,...$ and $\mathcal{G}^{n},$ $n=1,2,...$ are said to be \emph{%
asymptotically equivalent} if
\begin{equation*}
\Delta \left( \mathcal{E}^{n},\mathcal{G}^{n}\right) \rightarrow 0,\quad
\text{as}\quad n\rightarrow \infty ,
\end{equation*}
where $\Delta \left( \mathcal{E}^{n},\mathcal{G}^{n}\right) $ is the Le Cam
deficiency pseudo-distance between statistical experiments $\mathcal{E}^{n}$
and $\mathcal{G}^{n}.$
\end{definition}

From the above definitions it follows that, if in the sequence of models $%
\mathcal{E}^{n}$ there is a sequence of procedures $\pi _{1}^{n}$ such that
the risks $R(\mathcal{E}^{n},\pi _{1}^{n},L_{n},\theta )$ converge to the
quantity $\rho (\theta ),$ for a uniformy bounded sequence of loss functions
$L_{n},$ then there is a sequence of procedures $\pi _{2}^{n}$ in $\mathcal{G%
}^{n},$ such that the risks $R(\mathcal{G}^{n},\pi _{2}^{n},L_{n},\theta )$
converge to the same quantity $\rho (\theta ),$ uniformly in $\theta \in
\Theta .$ This indicates that the asymptotically minimax risks for bounded
loss functions in one model can be transferred to another model. In
particular one can compute the asymptotically minimax risk in non-Gaussian
models by computing it in the accompanying Gaussian models. This task,
however, remains beyond of the scope of the present paper.

\section{Formulation of the results\label{sect Res}}

\subsection{The parametric model}

Let $\Theta $ be an interval (possibly infinite) in the real line $\mathrm{R}
$ and
\begin{equation}
\mathcal{E}=\left( X,\mathcal{X},\{P_{\theta }:\theta \in \Theta \}\right)
\label{r-1-1}
\end{equation}
be a statistical experiment on the measurable space $\left( X,\mathcal{X}%
\right) $ with the parameter set $\Theta $ and a dominating $\sigma $-finite
measure $\mu .$ The last property means that for each $\theta \in \Theta $
the measure $P_{\theta }$ is absolutely continuous w.r.t. the measure $\mu .$
Denote by $p\left( x,\theta \right) $ the Radon-Nikodym derivative of $%
P_{\theta }$ w.r.t. $\mu :$%
\begin{equation}
p\left( x,\theta \right) =\frac{P_{\theta }\left( dx\right) }{\mu \left(
dx\right) },\quad x\in X,\quad \theta \in \Theta .  \label{r-1-2}
\end{equation}
For the sake of brevity we set $p(\theta )=p(\cdot ,\theta ).$ We shall
assume in the sequel that for any $\theta \in \Theta $%
\begin{equation}
p\left( \theta \right) >0,\quad \mu \text{-a.s. on }X.  \label{r-1-3}
\end{equation}
Assumption (\ref{r-1-3}) implies that the measures $P_{\theta },$ $\theta
\in \Theta $ are equivalent: $P_{\theta }\sim P_{u},$ for $\theta ,u\in
\Theta .$

\subsection{The nonparametric model}

Set $T=[0,1].$ Let $\mathcal{F}_{0}=\Theta ^{T}$ be the set of all functions
on the unit interval $T=[0,1]$ with values in the interval $\Theta :$
\begin{equation*}
\mathcal{F}_{0}=\Theta ^{T}=\left\{ f:[0,1]\rightarrow \Theta \right\} .
\end{equation*}
Let $\mathcal{H}\left( {\beta },L\right) $ be a H\"{o}lder ball of functions
defined on $T$ and with values in $\mathrm{R},$ i.e. the set of functions $%
f:T\rightarrow \mathrm{R},$ which satisfy a H\"{o}lder condition with
exponent ${\beta >0}$ and constant $L>0:$%
\begin{equation*}
\left| f^{(\beta _{0})}\left( t\right) -f^{(\beta _{0})}\left( s\right)
\right| \leq L\left| t-s\right| ^{\beta _{1}},\quad t,s\in T,
\end{equation*}
where the nonnegative integer $\beta _{0}$ and the real $\beta _{1}\in (0,1]$
are such that $\beta _{0}+\beta _{1}=\beta ,$ and which also satisfy
\begin{equation*}
\left\| f\right\| _{\infty }\leq L\text{, where }\left\| f\right\| _{\infty
}=\sup_{t\in T}\left| f\left( t\right) \right| .
\end{equation*}
Set for brevity, $\Sigma ^{\beta }=\mathcal{F}_{0}\cap \mathcal{H}\left( {%
\beta },L\right) .$ In the nonparametrically driven model to be treated here
it is assumed that we observe a sequence of independent r.v.'s $%
X_{1},...,X_{n}$ with values in the measurable space $\left( X,\mathcal{X}%
\right) ,$ such that, for each $i=1,...,n,$ the observation $X_{i}$ has
density $p\left( x,f\left( i/n\right) \right) $ where the function $f$ is
unknown and satisfies the smoothness condition $f\in \Sigma ^{\beta }.$ We
shall make use of the notation $P_{f}^{n}=P_{f(1/n)}\times ...\times
P_{f(n/n)},$ where $P_{\theta }$ is the probability measure in the
parametric experiment $\mathcal{E}$ and $f\in \Sigma ^{\beta }.$

\subsection{Regularity assumptions}

Assume that ${\beta }>1/2.$ In the sequel the density $p(x,\theta )$ in the
parametric experiment $\mathcal{E}$ shall be subjected to the regularity
assumptions $\left( R1\right) ,$ $\left( R2\right) ,$ $\left( R3\right) ,$
which are assumed to hold true with the same $\varepsilon >0.$

\begin{description}
\item[R1]  The function $s\left( \theta \right) =\sqrt{p\left( \theta
\right) }$ is smooth in the space $L^{2}\left( X,\mathcal{X},\mu \right) :$
there is a real number $\delta \in (\frac{1}{2{\beta }},1)$ and a map $%
\overset{\bullet }{s}\left( \theta \right) :\Theta \rightarrow L^{2}\left( X,%
\mathcal{X},\mu \right) $ such that
\begin{equation*}
\sup_{\left( \theta ,u\right) }\frac{1}{\left| u-\theta \right| ^{1+\delta }}%
\left( \int_{X}\left( s\left( x,u\right) -s\left( x,\theta \right) -\left(
u-\theta \right) \overset{\bullet }{s}\left( x,\theta \right) \right)
^{2}\mu \left( dx\right) \right) ^{1/2}<\infty ,
\end{equation*}
where $\sup $ is taken over all pairs $\left( \theta ,u\right) $ satisfying $%
\theta ,u\in \Theta ,$ $\left| u-\theta \right| \leq \varepsilon .$
\end{description}

It is well-known (see Strasser \cite{Str85}) that there is a map $\overset{%
\bullet }{l}\left( \theta \right) \in L^{2}\left( X,\mathcal{X},\mu \right) $
such that the function $\overset{\bullet }{s}\left( \theta \right) $ in
condition $\left( R1\right) $ can be written as
\begin{equation}
\overset{\bullet }{s}\left( \theta \right) =\frac{1}{2}\overset{\bullet }{l}%
\left( \theta \right) \sqrt{p\left( \theta \right) },\quad \mu \text{-a.s.
on }X.  \label{tan-vect}
\end{equation}
Moreover, $\overset{\bullet }{l}\left( \theta \right) \in L^{2}\left( X,%
\mathcal{X},P_{\theta }\right) $ and
\begin{equation*}
\text{E}_{\theta }\overset{\bullet }{l}\left( \theta \right) =\int_{X}%
\overset{\bullet }{l}\left( x,\theta \right) p\left( x,\theta \right) \mu
\left( dx\right) =0,\quad \theta \in \Theta ,
\end{equation*}
where $\text{E}_{\theta }$ is the expectation under $P_{\theta }.$ The map $%
\overset{\bullet }{l}\left( \theta \right) $ is called \emph{tangent vector}
at $\theta .$ For any $\theta \in \Theta ,$ define an \emph{extended tangent
vector} $\overset{\bullet }{l}_{\theta }\left( u\right) ,$ $u\in \Theta ,$
by setting
\begin{equation}
\overset{\bullet }{l}_{\theta }\left( x,u\right) =\left\{
\begin{array}{ll}
\overset{\bullet }{l}\left( x,\theta \right) , & \text{if\ }u=\theta , \\
\frac{2}{u-\theta }\left( \sqrt{\frac{p\left( x,u\right) }{p\left( x,\theta
\right) }}-1\right) , & \text{if\ }u\neq \theta .
\end{array}
\right.  \label{EXTEND-TG-V}
\end{equation}

\begin{description}
\item[R2]  There is a real number $\delta \in (\frac{2{\beta }+1}{2{\beta }-1%
},\infty )$ such that
\begin{equation*}
\sup_{\left( \theta ,u\right) }\int_{X}\left| \overset{\bullet }{l}_{\theta
}\left( x,u\right) \right| ^{2\delta }p\left( x,\theta \right) \mu \left(
dx\right) <\infty ,
\end{equation*}
where $\sup $ is taken over all pairs $\left( \theta ,u\right) $ satisfying $%
\theta ,u\in \Theta ,$ $\left| u-\theta \right| \leq \varepsilon .$
\end{description}

The Fisher information in the local experiment ${\mathcal{E}}$ is
\begin{equation}
I\left( \theta \right) =\int_X\left( \overset{\bullet }{l}\left( x,\theta
\right) \right) ^2p\left( x,\theta \right) \mu \left( dx\right) ,\quad
\theta \in \Theta .  \label{Fisher-inf}
\end{equation}

\begin{description}
\item[R3]  There are two real numbers $I_{\min }$ and $I_{\max }$ such that
\begin{equation*}
0<I_{\min }\leq I\left( \theta \right) \leq I_{\max }<\infty ,\quad \theta
\in \Theta .
\end{equation*}
\end{description}

\subsection{Local result}

First we state a local Gaussian approximation. For any $f\in \Sigma ^{\beta
},$ denote by $\Sigma _{f}^{\beta }\left( r\right) $ the neighborhood of $f,$
shifted to the origin:
\begin{equation*}
\Sigma _{f}^{\beta }\left( r\right) =\left\{ h:\left\| h\right\| _{\infty
}\leq r,\;f+h\in \Sigma ^{\beta }\right\} .
\end{equation*}
Set
\begin{equation*}
\overline{\gamma }_{n}=c\left( {\beta }\right) \left( \frac{\log n}{n}%
\right) ^{\frac{\beta }{2{\beta }+1}},
\end{equation*}
where $c\left( {\beta }\right) $ is a constant depending on ${\beta }.$
Throughout the paper $\gamma _{n}$ will denote a sequence of real numbers
satisfying, for some constant $c\geq 0,$%
\begin{equation}
\overline{\gamma }_{n}\leq \gamma _{n}=O(\overline{\gamma }_{n}\log ^{c}n).
\label{nonpar-rate}
\end{equation}
By definition the \emph{local experiment}
\begin{equation*}
\mathcal{E}_{f}^{n}=\left( X^{n},\mathcal{X}^{n},\{P_{f+h}^{n}:h\in \Sigma
_{f}^{\beta }(\gamma _{n})\}\right)
\end{equation*}
is generated by the sequence of independent r.v.'s $X_{1},...,X_{n},$ with
values in the measurable space $\left( X,\mathcal{X}\right) ,$ such that for
each $i=1,...,n,$ the observation $X_{i}$ has density $p\left( x,g\left(
i/n\right) \right) ,$ where $g=f+h,$ $h\in \Sigma _{f}^{\beta }\left( \gamma
_{n}\right) .$ An equivalent definition is:
\begin{equation}
\mathcal{E}_{f}^{n}=\mathcal{E}_{f}^{\left( 1\right) }\otimes ...\otimes
\mathcal{E}_{f}^{\left( n\right) },  \label{F-1}
\end{equation}
where
\begin{equation}
\mathcal{E}_{f}^{\left( i\right) }=\left( X,\mathcal{X},\{P_{g\left(
i/n\right) }:g=f+h,\;h\in \Sigma _{f}^{\beta }\left( \gamma _{n}\right)
\}\right) ,\quad i=1,...,n.  \label{F-2}
\end{equation}

\begin{theorem}
\label{Theorem R-2}Let ${\beta }>1/2$ and $I\left( \theta \right) $ be the
Fisher information in the parametric experiment $\mathcal{E}.$ Assume that
the density $p(x,\theta )$ satisfies the regularity conditions $\left(
R1-R3\right) .$ For any $f\in \Sigma ^{\beta },$ let $\mathcal{G}_{f}^{n}$
be the local Gaussian experiment, generated by observations
\begin{equation*}
Y_{i}^{n}=h\left( i/n\right) +I\left( f\left( i/n\right) \right)
^{-1/2}\varepsilon _{i},\quad i=1,...,n,
\end{equation*}
with $h\in \Sigma _{f}^{\beta }\left( r_{n}\right) ,$ where $\varepsilon
_{1},...,\varepsilon _{n}$ is a sequence of i.i.d. standard normal r.v.'s.
Then, uniformly in $f\in \Sigma ^{\beta },$ the sequence of local
experiments $\mathcal{E}_{f}^{n},$ $n=1,2,...$ is asymptotically equivalent
to the sequence of local Gaussian experiments $\mathcal{G}_{f}^{n},$ $%
n=1,2,...:$%
\begin{equation*}
\sup_{f\in \Sigma ^{\beta }}\Delta \left( \mathcal{E}_{f}^{n},\mathcal{G}%
_{f}^{n}\right) \rightarrow 0,\quad \text{as}\quad n\rightarrow \infty .
\end{equation*}
\end{theorem}

\subsection{Global result}

By definition the \emph{global experiment}
\begin{equation*}
\mathcal{E}^{n}=\left( X^{n},\mathcal{X}^{n},\{P_{f}^{n}:f\in \Sigma ^{\beta
}\}\right)
\end{equation*}
is generated by the sequence of independent r.v.'s $X_{1},...,X_{n},$ with
values in the measurable space $\left( X,\mathcal{X}\right) ,$ such that for
each $i=1,...,n,$ the observation $X_{i}$ has density $p\left( x,f\left(
i/n\right) \right) ,$ where $f\in \Sigma ^{\beta }.$\emph{\ }In other words $%
\mathcal{E}^{n}$ is the product experiment
\begin{equation*}
\mathcal{E}^{n}=\mathcal{E}^{\left( 1\right) }\otimes ...\otimes \mathcal{E}%
^{\left( n\right) },
\end{equation*}
where
\begin{equation*}
\mathcal{E}^{\left( i\right) }=\left( X,\mathcal{X},\{P_{f\left( i/n\right)
}:f\in \Sigma ^{\beta }\}\right) ,\quad i=1,...,n.
\end{equation*}

We shall make the following assumptions:

\begin{description}
\item[G1]  For any $\beta >1/2,$ there is an estimator $\widehat{f}%
_{n}:X^{n}\rightarrow \Sigma ^{\beta },$ such that
\begin{equation*}
\sup_{f\in \Sigma ^{\beta }}P\left( \left\| \widehat{f}_{n}-f\right\|
_{\infty }\geq \gamma _{n}\right) \rightarrow 0,\quad n\rightarrow \infty .
\end{equation*}

\item[G2]  The Fisher information $I(\theta ):\Theta \rightarrow (0,\infty )$
satisfies a H\"{o}lder condition with exponent $\alpha \in (1/2\beta ,1).$
\end{description}

The main result of the paper is the following theorem, which states a global
Gaussian approximation for the sequence of experiments $\mathcal{E}^{n},$ $%
n=1,2,...$ in the sense of Definition \ref{Definition-AE}.

\begin{theorem}
\label{Theorem R-1}Let ${\beta }>1/2$ and $I\left( \theta \right) $ be the
Fisher information in the parametric experiment $\mathcal{E}.$ Assume that
the density $p(x,\theta )$ satisfies the regularity conditions (R1-R3) and
that conditions (G1-G2) are fulfilled. Let $\mathcal{G}^{n}$ be the Gaussian
experiment generated by observations
\begin{equation*}
Y_{i}^{n}=\Gamma \left( f\left( i/n\right) \right) +\varepsilon _{i},\quad
i=1,...,n,
\end{equation*}
with $f\in \Sigma ^{\beta },$ where $\varepsilon _{1},...,\varepsilon _{n}$
is a sequence of i.i.d. standard normal r.v.'s and $\Gamma \left( \theta
\right) :\Theta \rightarrow \mathrm{R}$ is any function satisfying $\Gamma
^{\prime }\left( \theta \right) =\sqrt{I\left( \theta \right) }.$ Then the
sequence of experiments $\mathcal{E}^{n},$ $n=1,2,...$ is asymptotically
equivalent to the sequence of Gaussian experiments $\mathcal{G}^{n},$ $%
n=1,2,...:$%
\begin{equation*}
\Delta \left( \mathcal{E}^{n},\mathcal{G}^{n}\right) \rightarrow 0,\quad
\text{as}\quad n\rightarrow \infty .
\end{equation*}
\end{theorem}

\begin{remark}
Examples in Efromovich and Samarov \cite{Efr-Sam}, Brown and Zhang \cite
{Br-Zh} [see also Brown and Low \cite{Br-Low}] show that asymptotic
equivalence, in general, fails to hold true when $\beta \leq 1/2.$
\end{remark}

\begin{remark}
Assumption (G1) is related to attainable rates of convergence in the
sup-norm $\left\| \cdot \right\| _{\infty }$ for nonparametric regression.
It is well known that for a parameter space $\mathcal{H}\left( {\beta }%
,L\right) $, $\beta \leq 1$ the attainable rate for estimators of $f$ is $%
\left( \log n/n\right) ^{\beta /(2\beta +1)}$ in regular cases (cf. Stone
\cite{ston}). For a choice $\gamma _{n}=\bar{\gamma}_{n}\log ^{c}n$
condition (G1) would be implied by this type of result. However for a choice
$\gamma _{n}=\bar{\gamma}_{n}=c(\beta )\left( \left( \log n/n\right) ^{\beta
/(2\beta +1)}\right) $  (G1) is slightly stronger (the classical rate result
would also require $c(\beta )\rightarrow \infty $ for convergence to $0$ in
(G1)). In the case of the Gaussian location-type regression ((\ref{LRG-1})
for normal $\xi _{i}$) this is a consequence of the optimal constant result
of Korostelev \cite{koro}. The extension to our nongaussian regression
models would be of technical nature; for the density estimation model it has
been verified in Korostelev and Nussbaum \cite{k-n} and applied in a similar
context to here in Lemma 9.3 of \cite{Nuss} .
\end{remark}

\begin{remark}
\label{REMARK-VS_TR}The function $\Gamma \left( \theta \right) $ can be
related to so called \emph{variance-stabilizing} transformation, which we
proceed to introduce. Let $X_{1},...,X_{n}$ be a sequence of real valued
i.i.d. r.v.'s, with law depending on the parameter $\theta \in \Theta .$ Let
$\mu \left( \theta \right) $ be the common mean and $\sigma (\theta )$ be
the common variance. By the central limit theorem,
\begin{equation*}
\sqrt{n}\left\{ S_{n}-\mu (\theta )\right\} \overset{d}{\rightarrow }N\left(
0,\sigma \left( \theta \right) \right) ,
\end{equation*}
where $S_{n}=\left( X_{1}+...+X_{n}\right) /n.$ The variance-stabilizing
transformation is defined to be a function $F$ on the real line, with the
property that the variance of the limiting normal law does not depend on $%
\theta $, i.e. \
\begin{equation*}
\sqrt{n}\left\{ F\left( S_{n}\right) -F(\mu (\theta ))\right\} \overset{d}{%
\rightarrow }N\left( 0,1\right) .
\end{equation*}

The function $F$ pertaining to our nonparametric model can be computed and
the relation to the function $\Gamma $ can be clarified in some particular
cases. Let $\sigma (\theta )=I(\theta )>0$ and $\mu \left( \theta \right) $
satisfy $\mu ^{\prime }\left( \theta \right) =I\left( \theta \right) .$ Let $%
a\left( \lambda \right) $ be the inverse of the strictly increasing function
$\mu \left( \theta \right) $ on the interval $\Theta $ and $\mu \left(
\Theta \right) $ be its range. Assume that the $X_{i}$ take values in $\mu
(\Theta ).$ One can easily see that a variance stabilizing transformation is
any function $F(\lambda ),$ $\lambda \in \mu (\Theta )$ satisfying $%
F^{\prime }\left( \lambda \right) =1/\sqrt{I\left( a\left( \lambda \right)
\right) },$ $\lambda \in \mu (\Theta ).$ In this case, our transformation of
the functional parameter $f$ is actually the transformation of the mean $\mu
(\theta )$%
\begin{equation*}
\Gamma (\theta )=F(\mu (\theta )),
\end{equation*}
corresponding to this variance stabilizing transformation. In the other
particular case, when the mean of $X_{i}$ is $\mu (\theta )=\theta ,$ both
transformations $\Gamma $ and $F$ just coincide: $\Gamma (\theta )=F(\theta
).$
\end{remark}

We follow the line developed in Nussbaum \cite{Nuss} (see also Grama and
Nussbaum \cite{Gr-Nu-2}). The proof of Theorem \ref{Theorem R-2} is given in
Section \ref{SECTION-Local} and contains three steps:

\begin{itemize}
\item  Decompose the local experiments $\mathcal{E}_{f}^{n}$ and $\mathcal{G}%
_{f}^{n}$ into products of independent experiments:
\begin{equation*}
\mathcal{E}_{f}^{n}=\mathcal{E}_{f}^{n,1}\otimes ...\otimes \mathcal{E}%
_{f}^{n,m},\quad \mathcal{E}_{f}^{n}=\mathcal{G}_{f}^{n,1}\otimes ...\otimes
\mathcal{G}_{f}^{n,m}.
\end{equation*}
Here $m=o(n)$ and $\mathcal{E}_{f}^{n,k}$ represents the $k$-th ''block'' of
observations of size (approximately) $n/m$.

\item  Show that each component $\mathcal{E}_{f}^{n,k}$ can be well
approximated by its Gaussian counterpart $\mathcal{G}_{f}^{n,k}$ in the
sense that there exist equivalent versions $\widetilde{\mathcal{E}}%
_{f}^{n,k} $ and $\widetilde{\mathcal{G}}_{f}^{n,k}$ on a common measurable
space such that
\begin{equation*}
H^{2}\left( \widetilde{P}_{f,h}^{n,k},\widetilde{Q}_{f,h}^{n,k}\right)
=o\left( m^{-1}\right) ,
\end{equation*}
where $H^{2}\left( \cdot ,\cdot \right) $ is the Hellinger distance between
the probability measures $\widetilde{P}_{f,h}^{n,k}$ and $\widetilde{Q}%
_{f,h}^{n,k}$ in the local experiments $\widetilde{\mathcal{E}}_{f}^{n,k}$
and $\widetilde{\mathcal{G}}_{f}^{n,k}\mathcal{\ }$respectively.

\item  Patch together the components $\widetilde{\mathcal{E}}_{f}^{n,k}$ and
$\widetilde{\mathcal{G}}_{f}^{n,k},$ $k=1,...,m$ by means of the convenient
property of the Hellinger distance for the product of probability measures
\begin{equation*}
H^{2}\left( \widetilde{P}_{f,h}^{n},\widetilde{Q}_{f,h}^{n}\right) \leq
\sum_{k=1}^{m}H^{2}\left( \widetilde{P}_{f,h}^{n,k},\widetilde{Q}%
_{f,h}^{n,k}\right) =\sum_{i=1}^{m}o\left( m^{-1}\right) =o\left( 1\right) .
\end{equation*}
\end{itemize}

The challenge here is the second step. For its proof, in Section \ref
{Sect-gen-appr}, we develop a general approach, according to which any
experiment $\mathcal{E}^{n}$ with a certain asymptotic expansion of its
log-likelihood ratio (condition LASE) can be constructed on the same
measurable space with a Gaussian experiment $\mathcal{G}^{n},$ such that the
Hellinger distance between the corresponding probability measures converges
to $0$ at a certain rate. Then we are able to check condition LASE for the
model under consideration using a strong approximation result (see Theorem
\ref{THEOREM-HUNG-CONSTR}).

Theorem \ref{Theorem R-1} is derived from the local result of Theorem \ref
{Theorem R-2} by means of a globalizing procedure, which is presented in
Section \ref{SECTION-Global}.

\section{Examples\label{Sect-EX}}

\subsection{Location type regression model}

Consider the regression model with non-Gaussian additive noise
\begin{equation}
X_{i}=f(i/n)+\xi _{i},\quad i=1,...,n,  \label{EX-L-1}
\end{equation}
where $\xi _{1},...,\xi _{n}$ are i.i.d. r.v.'s of means $0$ and finite
variances, with density $p(x)$ on the real line, $f\in \Sigma ^{\beta }$ and
$\Sigma ^{\beta }$ is a H\"{o}lder ball on $[0,1]$ with exponent $\beta >1/2.
$ This model is a particular case of the nonparametrically driven model,
introduced in Section \ref{sect Res}, when $p(x,\theta )=p(x-\theta )$ is a
shift parameter family and $\theta \in \mathrm{R}.$ Assume that the
derivative $p^{\prime }(x)$ exists, for all $x\in \mathrm{R}.$ Then the
extended tangent vector, defined by (\ref{tan-vect}), is computed as follows
\begin{equation*}
\overset{\bullet }{l}_{\theta }(x,u)=\left\{
\begin{array}{ll}
\frac{p^{\prime }(x-\theta )}{p(x-\theta )}, & \quad \text{if}\quad u=\theta
, \\
\frac{2}{u-\theta }\frac{\sqrt{p(x-u)}-\sqrt{p(x-\theta )}}{\sqrt{p(x-\theta
)}}, & \quad \text{if}\quad u\neq \theta .
\end{array}
\right.
\end{equation*}
Set $s\left( x\right) =\sqrt{p(x)}.$ Then it is easy to see that conditions
(R1-R3) hold true, if we assume that the following items are satisfied:

\begin{description}
\item[L1]  The function $s^{\prime }\left( x\right) $ satisfies H\"{o}lder's
condition with exponent $\alpha ,$ where $\alpha \in (1/2\beta ,1),$ i.e.
\begin{equation*}
\left| s^{\prime }\left( x\right) -s^{\prime }\left( y\right) \right| \leq
C\left| x-y\right| ^{\alpha },\quad x,y\in \mathrm{R}.
\end{equation*}

\item[L2]  For some $\delta >\frac{2\beta +1}{2\beta -1}$ and $\varepsilon
>0,$%
\begin{equation*}
\sup_{|u|\leq \varepsilon }\int_{-\infty }^{\infty }\left| \frac{s^{\prime
}(x+u)}{s(x)}\right| ^{2\delta }p(x)dx<\infty .
\end{equation*}

\item[L3]  The Fisher informational number is positive:
\begin{equation*}
I=4\int_{-\infty }^{\infty }s^{\prime }(x)^{2}dx=\int_{-\infty }^{\infty }%
\frac{p^{\prime }(x)^{2}}{p(x)}dx>0.
\end{equation*}
\end{description}

It is well-known that a preliminary estimator satisfying condition (G1)
exists in this model. Then, according to Theorem \ref{Theorem R-1}, under
conditions (L1-L3) the model defined by observations (\ref{EX-L-1}) is
asymptotically equivalent to a linear regression with Gaussian noise, in
which we observe
\begin{equation*}
Y_{i}=f(i/n)+I^{-1/2}\varepsilon _{i},
\end{equation*}
where $\varepsilon _{1},...,\varepsilon _{n}$ are i.i.d. standard normal
r.v.'s.

\subsection{Exponential family model}

Another particular case of the model introduced in Section \ref{sect Res}
arises when the parametric experiment $\mathcal{E}=(X,\mathcal{X}%
,\{P_{\theta }:\theta \in \Theta \})$ is an one-dimensional linearly indexed
exponential family, where $\Theta $ is a possibly infinite interval on the
real line (see Brown \cite{Br}). This means that the measures $P_{\theta }$
are absolutely continuous w.r.t. a $\sigma $-finite measure $\mu (dx)$ with
densities (in the canonical form)
\begin{equation}
p(x,\theta )=\frac{P_{\theta }(dx)}{\mu (dx)}=\exp \left( \theta
U(x)-V(\theta )\right) ,\quad x\in X,\quad \theta \in \Theta ,
\label{EX-EXP-1}
\end{equation}
where the measurable function $U(x):X\rightarrow \mathrm{R}$ is a sufficient
statistic in the experiment $\mathcal{E}$ and
\begin{equation*}
V(\theta )=\log \int_{X}\exp \left( \theta U(x)\right) \mu (dx)
\end{equation*}
is the logarithm of the Laplace transformation of $U(x).$ It is easy to see
that regularity conditions (R1-R3) and (G1-G2) are satisfied, if we only
assume that $0<c_{1}\leq V^{\prime \prime }(\theta )\leq c_{2}<\infty $ and $%
\left| V^{\left( k\right) }(\theta )\right| \leq c_{2},$ for any $k\geq 3$
and two absolute constants $c_{1}$ and $c_{2}$ (for other related conditions
see Grama and Nussbaum \cite{Gr-Nu-2}). We point out that, for all these
models, the preliminary estimator of condition (G1) above can easily be
constructed (see for instance \cite{Gr-Nu-2}).

Now we shall briefly discuss some examples. Note that the parametrizations
in these is different from the canonical one appearing in (\ref{EX-EXP-1}).
We have chosen the natural parametrizations, where an observation $X$ in the
experiment $\mathcal{E}$ has mean $\mu (\theta )=\theta ,$ since this
facilitates computation of the function $\Gamma (\theta ).$

\emph{Gaussian scale model. }Assume that we are given a sequence of normal
observations $X_{1},...,X_{n}$ with mean $0$ and variance $f(i/n),$ where
the function $f(t),$ $t\in \lbrack 0,1]$ satisfies a H\"{o}lder condition
with exponent $\beta >1/2$ and is such that $c_{1}\leq f(t)\leq c_{2},$ for
some positive absolute constants $c_{1}$ and $c_{2}.$ In this model the
density of the observations has the form $p(x,\theta )=\frac{1}{\sqrt{2\pi }%
\theta }\exp \left( -\frac{x^{2}}{2\theta ^{2}}\right) ,$ $x\in \mathrm{R}$
and the Fisher information is $I(\theta )=2\theta ^{-2}.$ This gives us $%
\Gamma (\theta )=\sqrt{2}\log \theta .$ Then, by Theorem \ref{Theorem R-1},
the model is asymptotically equivalent to the Gaussian model
\begin{equation*}
Y_{i}=\sqrt{2}\log f(i/n)+\varepsilon _{i},
\end{equation*}
where $\varepsilon _{1},...,\varepsilon _{n}$ are i.i.d. standard normal
r.v.'s.

\emph{Poisson model.} Assume that we are given a sequence of Poisson
observations $X_{1},...,X_{n}$ with parameters $f(i/n),$ where the function $%
f(t),$ $t\in \lbrack 0,1]$ satisfies a H\"{o}lder condition with exponent $%
\beta >1/2$ and is such that $c_{1}\leq f(t)\leq c_{2},$ for some positive
absolute constants $c_{1}$ and $c_{2}.$ In this model $p(x,\theta )=\theta
^{x}\exp \left( -\theta \right) ,$ $x\in X=\{0,1,...\}$ and $I(\theta
)=\theta ^{-1}.$ As a consequence $\Gamma (\theta )=2\sqrt{\theta }.$
According to Theorem \ref{Theorem R-1}, the model is asymptotically
equivalent to the Gaussian model
\begin{equation*}
Y_{i}=2\sqrt{f(i/n)}+\varepsilon _{i},
\end{equation*}
where $\varepsilon _{1},...,\varepsilon _{n}$ are i.i.d. standard normal
r.v.'s.

\emph{Binary response model.} Assume that we are given a sequence of
Bernoulli observations $X_{1},...,X_{n}$ taking values $0$ and $1$ with
probabilities $1-f(i/n)$ and $f(i/n)$ respectively, where the function $f(t),
$ $t\in \lbrack 0,1]$ satisfies a H\"{o}lder condition with exponent $\beta
>1/2$ and is such that $c_{1}\leq f(t)\leq c_{2}$ for some absolute
constants $c_{1}>0$ and $c_{2}<1.$ In this model $p(x,\theta )=\theta
^{x}(1-\theta )^{1-x},$ $x\in X=\{0,1\}$ and $I(\theta )=\frac{1}{\theta
(1-\theta )}.$ This yields $\Gamma (\theta )=2\arcsin \sqrt{\theta }.$ By
Theorem \ref{Theorem R-1}, this model is asymptotically equivalent to the
Gaussian model
\begin{equation*}
Y_{i}=2\arcsin \sqrt{f(i/n)}+\varepsilon _{i},
\end{equation*}
where $\varepsilon _{1},...,\varepsilon _{n}$ are i.i.d. standard normal
r.v.'s.

\section{Some nonparametric local asymptotic theory\label{Sect-gen-appr}}

The aim of this section is to state and discuss a condition on the
asymptotic behaviour of the likelihood ratio which can be used for Gaussian
approximation of our nonparametric regression models. This condition will
resemble an LAN-condition, but will be stronger in the sense that it
requires an asymptotic expansion of the log-likelihood ratio in terms of
independent random variables. Although LAN conditions for nonparametric
experiments have been stated and their consequences been developed
extensively, (see for instance Millar \cite{Mill}, Strasser \cite{Str85} and
references therein), the context of this was $n^{-1/2}$-consistent
estimation problems, where an $n^{-1/2}$-localization of the experiments is
useful and implies weak convergence of the sequence $\mathcal{E}^{n}$ to a
limit Gaussian experiment $\mathcal{G}.$ In contrast, we assert the
existence of a version of the experiment $\mathcal{E}^{n}$ on the same
sample space with a suitable Gaussian experiment $\mathcal{G}^{n}$ such that
the Hellinger distance between corresponding measures goes to $0$ at some
rate. The existence of a limit experiment $\mathcal{G}$ is not assumed here:
it is replaced by a suitable sequence of approximating experiments $\mathcal{%
G}^{n}.$

It should be mentioned that the scope of applicability of this theory is
actually larger than exploited here: it can be used to establish asymptotic
equivalence results for regression models with i.i.d. observations, such as
models with random design, and for models with dependent observations.
However these results remain beyond of the scope of the present paper. Here
we restrict ourselves to the model introduced in Section \ref{sect Res}.

\subsection{A bound for the Hellinger distance}

First we shall give sufficient conditions for a rate of convergence to $0$
of the Hellinger distance between the corresponding measures of two
statistical experiments.

Assume that $\mathcal{F}$ is an arbitrary set and that for any $f\in
\mathcal{F}$ we are given a system of sets $\mathcal{F}_{f}\left( r\right) ,$
$r>0,$ to be regarded as a system of neighborhoods of $f.$ Let $r_{n},$ $%
n=1,2,...$ be a sequence of real numbers satisfying $0<r_{n}\leq 1$ and $%
r_{n}\rightarrow 0,$ as $n\rightarrow \infty .$ Let, for any $n=1,2,...\;$%
and any $f\in \mathcal{F},$%
\begin{equation*}
\mathcal{E}_{f}^{n}=\left( X^{n},\mathcal{X}^{n},\left\{ P_{f,h}^{n}:h\in
\mathcal{F}_{f}\left( r_{n}\right) \right\} \right)
\end{equation*}
and
\begin{equation*}
\mathcal{G}_{f}^{n}=\left( X^{n},\mathcal{X}^{n},\left\{ Q_{f,h}^{n}:h\in
\mathcal{F}_{f}\left( r_{n}\right) \right\} \right)
\end{equation*}
be two statistical experiments with the same sample space $\left( X^{n},%
\mathcal{X}^{n}\right) $ and the same parameter space $\mathcal{F}_{f}\left(
r_{n}\right) .$ Assume that the ''local'' experiments $\mathcal{E}_{f}^{n}$
and $\mathcal{G}_{f}^{n}$ have a common ''central'' measure, i.e. that there
is an element $h_{0}=h_{0}\left( f\right) \in \mathcal{F}_{f}\left(
r_{n}\right) $ such that $P_{f,h_{0}}^{n}=Q_{f,h_{0}}^{n}=\mathbf{P}_{f}^{n}$
and $P_{f,h}^{n}\ll P_{f,h_{0}}^{n},$ $Q_{f,h}^{n}\ll Q_{f,h_{0}}^{n},$ for
any $h\in \mathcal{F}_{f}\left( r_{n}\right) .$

\begin{theorem}
\label{Theorem-G-1}Let $\alpha _{1}>\alpha \geq 0.$ Assume that, for some $%
c_{1}>0,$
\begin{equation}
\sup_{f\in \mathcal{F}}\sup_{h\in \mathcal{F}_{f}\left( r_{n}\right) }%
\mathbf{P}_{f}^{n}\left( \left| \log \frac{dP_{f,h}^{n}}{dP_{f,h_{0}}^{n}}%
-\log \frac{dQ_{f,h}^{n}}{dQ_{f,h_{0}}^{n}}\right| \geq c_{1}r_{n}^{\alpha
_{1}}\right) =O\left( r_{n}^{2\alpha _{1}}\right)  \label{CC1}
\end{equation}
and, for any $\varepsilon \in (0,1),$
\begin{equation}
\sup_{f\in \mathcal{F}}\sup_{h\in \mathcal{F}_{f}\left( r_{n}\right)
}P_{f,h}^{n}\left( \log \frac{dP_{f,h}^{n}}{dP_{f,h_{0}}^{n}}>-\varepsilon
\log r_{n}\right) =O\left( r_{n}^{2\alpha _{1}}\right)  \label{CC2}
\end{equation}
and
\begin{equation}
\sup_{f\in \mathcal{F}}\sup_{h\in \mathcal{F}_{f}\left( r_{n}\right)
}Q_{f,h}^{n}\left( \log \frac{dQ_{f,h}^{n}}{dQ_{f,h_{0}}^{n}}>-\varepsilon
\log r_{n}\right) =O\left( r_{n}^{2\alpha _{1}}\right) .  \label{CC3}
\end{equation}
Then there is an $\alpha _{2}>\alpha $ such that
\begin{equation*}
\sup_{f\in \mathcal{F}}\sup_{h\in \mathcal{F}_{f}\left( r_{n}\right)
}H^{2}\left( P_{f,h}^{n},Q_{f,h}^{n}\right) =O\left( r_{n}^{2\alpha
_{2}}\right) .
\end{equation*}
\end{theorem}

\begin{proof}
Set, for brevity
\begin{equation*}
L_{f,h}^{1,n}=\frac{dP_{f,h}^n}{dP_{f,h_0}^n},\quad \quad L_{f,h}^{2,n}=%
\frac{dQ_{f,h}^n}{dQ_{f,h_0}^n},\quad \quad \Psi _n=\sqrt{L_{f,h}^{1,n}}-%
\sqrt{L_{f,h}^{2,n}}.
\end{equation*}
Consider  the set
\begin{equation*}
A_n=\left\{ \left| \log L_{f,h}^{1,n}-\log L_{f,h}^{1,n}\right| \leq
r_n^{\alpha _1}\right\} .
\end{equation*}
With these notations, by the definition of the Hellinger distance [see (\ref
{Hell-dist})], we have
\begin{equation}
H^2\left( P_{f,h}^n,Q_{f,h}^n\right) =\frac 12\mathbf{E}_f^n\Psi _n^2=\frac
12\mathbf{E}_f^n\mathbf{1}_{A_n}\Psi _n^2+\frac 12\mathbf{E}_f^n\mathbf{1}%
_{A_n^c}\Psi _n^2.  \label{hd-1}
\end{equation}
Changing  measure in the first expectation in the right hand side of (\ref
{hd-1}), we write
\begin{eqnarray}
\mathbf{E}_f^n\mathbf{1}_{A_n}\Psi _n^2 &=&\int_{\Omega ^n}\mathbf{1}%
_{A_n}\left( \sqrt{L_{f,h}^{1,n}/L_{f,h}^{2,n}}-1\right) ^2dQ_{f,h}^n  \notag
\\
&=&\int_{\Omega ^n}\mathbf{1}_{A_n}\left( \exp \left( \frac 12\log
L_{f,h}^{1,n}-\frac 12\log L_{f,h}^{2,n}\right) -1\right) ^2dQ_{f,h}^n
\notag \\
&\leq &\left( \exp \left( \frac 12r_n^{\alpha _1}\right) -1\right)
^2=O\left( r_n^{2\alpha _1}\right) .  \label{hd-2}
\end{eqnarray}
An application of the elementary inequality $\left( a+b\right) ^2\leq
2a^2+2b^2$ gives the bound for the second expectation in the right hand
side of (\ref{hd-1}):
\begin{equation}
\mathbf{E}_f^n\mathbf{1}_{A_n^c}\Psi _n^2\leq 2\mathbf{E}_f^n\mathbf{1}%
_{A_n^c}L_{f,h}^{1,n}+2\mathbf{E}_f^n\mathbf{1}_{A_n^c}L_{f,h}^{2,n}.
\label{hd-3}
\end{equation}
We proceed to estimate $\mathbf{E}_f^n\mathbf{1}_{A_n^c}L_{f,h}^{1,n}.$
Setting $B_n=\left\{ \log L_{f,h}^{1,n}\leq -\delta \log r_n\right\} ,$
where $\delta >0$ will be specified below, one gets
\begin{eqnarray}
\mathbf{E}_f^n\mathbf{1}_{A_n^c}L_{f,h}^{1,n} &=&\mathbf{E}_f^n\mathbf{1}%
_{A_n^c\cap B_n}L_{f,h}^{1,n}+\mathbf{E}_f^n\mathbf{1}_{A_n^c\cap
B_n^c}L_{f,h}^{1,n}  \notag \\
&\leq &r_n^{-\delta }\mathbf{P}_f^n\left( A_n^c\right) +\mathbf{E}_f^n%
\mathbf{1}_{B_n^c}L_{f,h}^{1,n}.  \label{hd-4}
\end{eqnarray}
Choosing$\ \delta $ small such that $\alpha _2=\alpha _1-\delta >\alpha ,$ we
get,
\begin{equation}
\sup_{f,h}n^\delta \mathbf{P}_f^n\left( A_n^c\right) =O\left( r_n^{\alpha
_1-\delta }\right) =O\left( r_n^{\alpha _2}\right) .  \label{hd-5}
\end{equation}
The second term on the right-hand side of (\ref{hd-4}) can be written as
\begin{eqnarray}
\mathbf{E}_f^n\mathbf{1}_{B_n^c}L_{f,h}^{1,n} &=&\mathbf{E}_f^n\left( \log
L_{f,h}^{1,n}>-\delta \log r_n\right) L_{f,h}^{1,n}  \notag \\
&=&P_{f,h}^n\left( \log L_{f,h}^{1,n}>-\delta \log r_n\right) =O\left(
r_n^{\alpha _1}\right) .  \label{hd-6-1}
\end{eqnarray}
Inserting  (\ref{hd-5}) and (\ref{hd-6-1}) in (\ref{hd-4}) we get
\begin{equation}
\mathbf{E}_f^n\mathbf{1}_{A_n^c}L_{f,h}^{1,n}=O\left( r_n^{\alpha _2}\right)
.  \label{hd-7}
\end{equation}
An estimate for the second term on the right-hand side of (\ref{hd-3}) is
proved in exactly the same way. This gives $\mathbf{E}_f^n\mathbf{1}%
_{A_n^c}\Psi _n^2=O\left( r_n^{\alpha _2}\right) ,$ which in turn, in conjunction
with (\ref{hd-2}) and (\ref{hd-1}), concludes the proof of Theorem \ref
{Theorem-G-1}.
\end{proof}

\subsection{Nonparametric experiments which admit a locally asymptotic
stochastic expansion}

We shall show that the assumptions in Theorem \ref{Theorem-G-1} are
satisfied if the log-likelihood ratio in the experiment $\mathcal{E}^{n}$
admits a certain stochastic expansion in terms of independent random
variables.

Let $T=[0,1]$ and $\mathcal{F}\subset \Theta ^{T}$ be some set of functions $%
f\left( t\right) :T\rightarrow \Theta .$ Let $\mathcal{E}^{n},$ $n\geq 1$ be
a sequence of statistical experiments
\begin{equation*}
\mathcal{E}^{n}=\left( \Omega ^{n},\mathcal{A}^{n},\left\{ P_{f}^{n}:f\in
\mathcal{F}\right\} \right) ,
\end{equation*}
with parameter set $\mathcal{F}.$ For simplicity we assume that, for any $%
n=1,2,...$ the measures $P_{f}^{n}:f\in \mathcal{F}$ in the experiment $%
\mathcal{E}^{n}$ are equivalent, i.e. that $P_{f}^{n}\ll P_{g}^{n},$ for any
$f,g\in \mathcal{F}.$ Recall that $\mathcal{H}(\beta ,L)$ is a H\"{o}lder
ball of functions defined on $T$ with values in $\Theta .$ The parameters $%
\beta $ and $L$ are assumed to be absolute constants satisfying $\beta >1/2$
and $0<L<\infty .$ It will be convenient to define the neighborhoods of $%
f\in \mathcal{F}$ (shifted to the origin) as follows: for any non-negative
real number $r,$
\begin{equation*}
\mathcal{F}_{f}(r)=\{rh:h\in \mathcal{H}(\beta ,L),\;f+rh\in \mathcal{F}\}.
\end{equation*}

For stating our definition this we need the following objects:

\begin{description}
\item[E1]  A sequence of real numbers $r_{n},$ $n=1,2,...$ which satisfies $%
r_{n}\rightarrow 0,$ as $n\rightarrow \infty .$

\item[E2]  A function $I\left( \theta \right) :\Theta \rightarrow (0,\infty
),$ which will play the role of the Fisher information in the experiment $%
\mathcal{E}^{n}.$

\item[E3]  The triangular array of design points $t_{ni}=i/n,$ $i=1,...,n,$ $%
n\geq 1,$ on the interval $T=[0,1]$.
\end{description}

\begin{definition}
\label{Def LAN}The sequence of experiments $\mathcal{E}^{n},$ $n\geq 1$ is
said to satisfy condition LASE with rate $r_{n}$ and local Fisher
information function $I\left( \cdot \right) ,$ if, for any fixed $n\geq 1$
and any fixed $f\in \mathcal{F},$ on the probability space $(\Omega ^{n},%
\mathcal{A}^{n},P_{f}^{n})$ there is a sequence of independent r.v.'s $\xi
_{ni}\left( f\right) ,$ $i=1,...,n$ of mean $0$ and variances
\begin{equation*}
E_{f}^{n}\xi _{ni}^{2}\left( f\right) =I\left( f\left( t_{ni}\right) \right)
,\quad i=1,...,n,
\end{equation*}
such that the expansion
\begin{equation*}
\log \frac{dP_{f+h}^{n}}{dP_{f}^{n}}=\sum_{i=1}^{n}h\left( t_{ni}\right) \xi
_{ni}\left( f\right) -\frac{1}{2}\sum_{i=1}^{n}h\left( t_{ni}\right)
^{2}I\left( f\left( t_{ni}\right) \right) +\rho _{n}\left( f,h\right) ,
\end{equation*}
holds true for any $h\in \mathcal{F}_{f}(r_{n}),$ where the remainder $\rho
_{n}\left( f,h\right) $ satisfies
\begin{equation*}
P_{f_{n}}^{n}\left( \left| \rho _{n}\left( f_{n},h_{n}\right) \right|
>a\right) \rightarrow 0,
\end{equation*}
for any two fixed sequences $\left( f_{n}\right) _{n\geq 1}$ and $\left(
h_{n}\right) _{n\geq 1}$ subject to $f_{n}\in \mathcal{F},$ $h_{n}\in
\mathcal{F}_{f_{n}}(r_{n})$ and any real $a>0,$ as $n\rightarrow \infty .$
\end{definition}

In the sequel we shall impose conditions (C1-C4) as formulated below.

\begin{description}
\item[C1]  The sequence $r_{n},$ $n=1,2,...$ has the parametric rate, i.e.
is so that
\begin{equation*}
r_{n}=c\frac{1}{\sqrt{n}}.
\end{equation*}

\item[C2]  The remainder $\rho _{n}\left( f,h\right) $ in the definition of
condition LASE converges to $0$ at a certain rate: for some ${\alpha }\in (%
\frac{1}{2{\beta }},1)$ and any $\varepsilon >0,$
\begin{equation*}
\sup_{f}\sup_{h}P_{f_{n}}^{n}\left( \left| \rho _{n}\left(
f_{n},h_{n}\right) \right| \geq \varepsilon n^{-\alpha /2}\right) =O\left(
n^{-\alpha }\right) ,
\end{equation*}
where $\sup $ is taken over all possible $f\in \mathcal{F}$ and $h\in
\mathcal{F}_{f}(r_{n}).$

\item[C3]  The r.v.'s $\xi _{ni}\left( f\right) ,$ $i=1,...,n$ in the
definition of condition LASE satisfy a strengthened version of the Lindeberg
condition: for some ${\alpha }\in (\frac{1}{2{\beta }},1)$ and any $%
\varepsilon >0,$%
\begin{equation*}
\sup_{f\in \mathcal{F}}\frac{1}{n}\sum_{i=1}^{n}E_{f}^{n}(n^{\alpha /2}\xi
_{ni}\left( f\right) )^{2}\mathbf{1}\left( \left| n^{\alpha /2}\xi
_{ni}\left( f\right) \right| \geq \varepsilon \sqrt{n}\right) =O\left(
n^{-\alpha }\right) .
\end{equation*}

\item[C4]  For $n=1,2,...$ the local Fisher information function $I\left(
\theta \right) $ satisfies
\begin{equation*}
0<I_{\min }\leq I\left( \theta \right) \leq I_{\max }<\infty ,\quad \theta
\in \Theta .
\end{equation*}
\end{description}

Let $\mathcal{E}_{f}^{n}$ be the local experiment
\begin{equation*}
\mathcal{E}_{f}^{n}=\left( \Omega ^{n},\mathcal{A}^{n},\left\{
P_{f+h}^{n}:h\in \mathcal{F}_{f}(r_{n})\right\} \right)
\end{equation*}
and let $f_{n},$ $n\geq 1$ denote any sequence of functions in $\mathcal{F}.$
Let $H(\cdot ,\cdot )$ be the Hellinger distance between probability
measures, see (\ref{Hell-dist}). The next theorem states that, under a
strengthened version of condition LASE, the sequence of local experiments $%
\mathcal{E}_{f_{n}}^{n},$ $n\geq 1$ can be approximated by a sequence of
Gaussian shift experiments uniformly in all sequences $f_{n},$ $n\geq 1,$
using the distance $H(\cdot ,\cdot ).$

\begin{theorem}
\label{Theorem LAQ-1}Assume that the sequence of experiments $\mathcal{E}%
^{n},$ $n=1,2,...$ satisfies condition LASE with rate $r_{n}$ and local
Fisher information function $I\left( \theta \right) $ and conditions (C1-C4)
hold true. Let $\varepsilon _{i},$ $i=1,2,...$be a sequence of i.i.d.
standard normal r.v.'s defined on a probability space $\left( \Omega ^{0},%
\mathcal{A}^{0},\mathbf{P}\right) $. Let, for any fixed $n\geq 1$ and fixed $%
f\in \Sigma ,$%
\begin{equation*}
\mathcal{G}_{f}^{n}=\left( \mathrm{R}^{n},\mathcal{B}^{n},\left\{
Q_{f,h}^{n}:h\in \mathcal{F}_{f}(r_{n})\right\} \right)
\end{equation*}
be the Gaussian shift experiment generated by $n$ observations
\begin{equation*}
Y_{i}^{n}=h\left( i/n\right) +\frac{1}{\sqrt{I\left( f\left( i/n\right)
\right) }}\varepsilon _{i},\quad i=1,...,n,
\end{equation*}
with $h\in \mathcal{F}_{f}(r_{n}).$ Then, for any fixed $n\geq 1$ and $f\in
\mathcal{F},$ there are experiments
\begin{eqnarray*}
\widetilde{\mathcal{G}}_{f}^{n} &=&\left( \Omega ^{0},\mathcal{A}%
^{0},\left\{ \widetilde{Q}_{f,h}^{n}:h\in \mathcal{F}_{f}(r_{n})\right\}
\right) , \\
\widetilde{\mathcal{E}}_{f}^{n} &=&\left( \Omega ^{0},\mathcal{A}%
^{0},\left\{ \widetilde{P}_{f,h}^{n}:h\in \mathcal{F}_{f}(r_{n})\right\}
\right)
\end{eqnarray*}
such that
\begin{equation*}
\Delta \left( \mathcal{G}_{f}^{n},\widetilde{\mathcal{G}}_{f}^{n}\right)
=\Delta \left( \mathcal{E}_{f}^{n},\widetilde{\mathcal{E}}_{f}^{n}\right) =0
\end{equation*}
and for some $\alpha \in \left( 1/2\beta ,1\right) ,$%
\begin{equation*}
\sup_{f\in \mathcal{F}}\sup_{h\in \mathcal{F}_{f}(r_{n})}H^{2}\left(
\widetilde{P}_{f,h}^{n},\widetilde{Q}_{f,h}^{n}\right) =O\left( r_{n}^{{%
2\alpha }}\right) ,
\end{equation*}
as $n\rightarrow \infty .$
\end{theorem}

We give here some hints how to carry out the proof of Theorem \ref{Theorem
LAQ-1}. Starting with the independent standard Gaussian sequence $%
\varepsilon _{i},$ $i=1,2,...$, we construct a sequence $\widetilde{\xi }%
_{n1},...,\widetilde{\xi }_{nn},\widetilde{\rho }_{n}\left( f,h\right) $
with the same joint distribution as $\xi _{n1},...,\xi _{nn},\rho _{n}\left(
f,h\right) $ from the expansion for the likelihood $%
L_{f,h}^{1,n}=dP_{f+h}^{n}/dP_{f}^{n}.$ This will ensure that the ''new''
likelihood $\widetilde{L}_{f,h}^{1,n}$, as a process indexed by $h\in
\mathcal{F}_{f}(r_{n})$, has the same law as $L_{f,h}^{1,n},$ and thus, that
the corresponding experiments are exactly equivalent. The key point in this
construction is to guarantee that the two length $n$ sequences $%
I^{1/2}\left( f\left( i/n\right) \right) $ $\varepsilon _{i},$ $i=1,...,n$
and $\widetilde{\xi }_{ni},$ $i=1,...,n$ are as close as possible. For this
we make use of a strong approximation result for partial sums of independent
r.v.'s indexed by functions, provided by Theorem \ref{THEOREM-HUNG-CONSTR}
(see the Appendix). We note also that the new remainder $\widetilde{\rho }%
_{n}\left( f,h\right) $ will satisfy the same requirements as $\rho
_{n}\left( f,h\right) $ does, since both are equally distributed.

\subsection{Asymptotic expansion with bounded scores.}

Assume that the sequence of experiments $\mathcal{E}^{n},$ $n=1,2,...$
satisfies condition LASE. This means that, for $f,$ $h$ satisfying $f\in
\mathcal{F}$ and $h\in \mathcal{F}_{f}(r_{n}),$
\begin{equation}
\log \frac{dP_{f+h}^{n}}{dP_{f}^{n}}=\sum_{i=1}^{n}h\left( t_{ni}\right) \xi
_{ni}-\frac{1}{2}\sum_{i=1}^{n}h\left( t_{ni}\right) ^{2}I\left( f\left(
t_{ni}\right) \right) +\rho _{n}\left( f,h\right) ,  \label{c-a-2}
\end{equation}
where $\xi _{ni}=\xi _{ni}(f),$ $i=1,...,n$ is a sequence of independent
r.v.'s of mean $0$ and variances $\text{E}_{f}^{n}\xi _{ni}^{2}=I\left(
f\left( t_{ni}\right) \right) .$ If the model is of location type
\begin{equation*}
X_{i}=f\left( i/n\right) +\eta _{i},\quad i=1,...,n,
\end{equation*}
where the noise $\eta _{i}$ has density $p(x)$ then $\xi _{ni}$ in (\ref
{c-a-2}) stands for $l(\eta _{i})$ where $l(x)=p^{\prime }(x)/p(x)$ is often
called the score function. In a somewhat loose terminology, the r.v.'s $\xi
_{ni}$ may therefore be called ''scores''. We shall show that, under the
conditions (C1-C4), the above expansion can be modified so that the r.v.'s $%
\xi _{ni}$ are replaced by some bounded r.v.'s $\xi _{ni}^{\ast }$ with the
same mean and variances. More precisely, we prove the following.

\begin{lemma}
\label{Lemma-modif-exp}Let conditions (C1-C4) hold true. Then there is a
sequence of independent r.v. $\xi _{ni}^{\ast }\left( f\right) ,$ $i=1,...,n$
of means $0$ and variances $E_{f}^{n}\xi _{ni}^{\ast }\left( f\right)
^{2}=I\left( f\left( t_{ni}\right) \right) ,$ $i=1,...,n,$ satisfying
\begin{equation}
\left| r_{n}^{1-\alpha }\xi _{ni}^{\ast }\left( f\right) \right| \leq
c,\quad i=1,...,n  \label{c-b-1}
\end{equation}
for some real number $\alpha \in \left( 1/2\beta ,1\right) ,$ and such that
\begin{equation}
\log \frac{dP_{f+h}^{n}}{dP_{f}^{n}}=\sum_{i=1}^{n}h\left( t_{ni}\right) \xi
_{ni}^{\ast }\left( f\right) -\frac{1}{2}\sum_{i=1}^{n}h\left( t_{ni}\right)
^{2}I\left( f\left( t_{ni}\right) \right) +\rho _{n}^{\ast }\left(
f,h\right) ,  \label{c-b-2}
\end{equation}
where, for any $\varepsilon >0,$
\begin{equation}
\sup_{f\in \mathcal{F}}\sup_{h\in \mathcal{F}_{f}\left( r_{n}\right)
}P_{f}^{n}\left( \left| \rho _{n}^{\ast }\left( f,h\right) \right| \geq
\varepsilon r_{n}^{\alpha }\right) =O\left( r_{n}^{2\alpha }\right) .
\label{c-b-2-1}
\end{equation}
\end{lemma}

\begin{proof}
Let $\alpha \in \left( \frac 1{2\beta },1\right) $ be the real number for
which assumptions (C2) and (C3) hold true. Since $E_f^n\xi _{ni}=0,$ we have
\begin{equation}
\xi _{ni}=\xi _{ni}^{\prime }+\xi _{ni}^{\prime \prime }=\eta _{ni}^{\prime
}+\eta _{ni}^{\prime \prime },\quad i=1,...,n,  \label{c-b-p1}
\end{equation}
where
\begin{equation}
\xi _{ni}^{\prime }=\xi _{ni}\mathbf{1}\left( \left| r_n\xi _{ni}\right|
\leq r_n^{{\alpha }}\right) ,\quad \xi _{ni}^{\prime \prime }=\xi _{ni}%
\mathbf{1}\left( \left| r_n\xi _{ni}\right| >r_n^{{\alpha }}\right)
\label{c-b-p2}
\end{equation}
and
\begin{equation}
\eta _{ni}^{\prime }=\xi _{ni}^{\prime }-E_f^n\xi _{ni}^{\prime },\quad \eta
_{ni}^{\prime \prime }=\xi _{ni}^{\prime \prime }-E_f^n\xi _{ni}^{\prime
\prime }.  \label{c-b-p2a}
\end{equation}
Set
\begin{equation}
v_{ni}^2=E_f^n\xi _{ni}^2-E_f^n(\eta _{ni}^{\prime })^2  \label{c-b-2-nn}
\end{equation}
and
\begin{equation*}
p_{ni}=\frac 12x_n^{-2}v_{ni}^2,\quad x_n=c_1r_n^{\alpha -1},
\end{equation*}
Since $r_n\rightarrow 0$ as $n\rightarrow \infty $ and $v_{ni}^2\leq
E_f^n\xi _{ni}^2=I\left( f\left( t_{ni}\right) \right) \leq I_{\max },$ the
constant $c_1$ can be chosen large enough so that $p_{ni}\leq 1/2$ for any $n\geq 1.$

Without loss of generality one may assume that on the probability space $%
\left( \Omega ^n,\mathcal{A}^n,P_f^n\right) $ there is a sequence of
independent r.v.'s $\eta _{ni}^{\prime \prime \prime },$ $i=1,...,n,$
independent of the sequence $\xi _{ni},$ $i=1,...,n,$ which take values $%
-x_n,$ $0,$ $x_n$ with probabilities $p_{ni},$ $1-2p_{ni},$ $p_{ni}$
respectively. It is clear that the r.v.'s $\eta _{ni}^{\prime \prime \prime
} $ are such that
\begin{equation}
\left| r_n^{1-\alpha }\eta _{ni}^{\prime \prime \prime }\right| \leq
c_1,\quad E_f^n\eta _{ni}^{\prime \prime \prime }=0,\quad E_f^n(\eta
_{ni}^{\prime \prime \prime })^2=v_{ni}^2.  \label{c-b-2-ee}
\end{equation}
Set
\begin{equation}
\xi _{ni}^{*}=\eta _{ni}^{\prime }+\eta _{ni}^{\prime \prime \prime },\quad
i=1,...,n.  \label{c-b-p3}
\end{equation}
>From this definition it is clear that (\ref{c-b-1}) holds true with $%
c=2+c_1. $ Since $\xi _{ni}$ and $\eta _{ni}^{\prime \prime \prime }$ are
independent, taking into account (\ref{c-b-2-ee}) and (\ref{c-b-2-nn}), we
get
\begin{equation*}
E_f^n(\xi _{ni}^{*})^2=E_f^n(\eta _{ni}^{\prime })^2+E_f^n(\eta
_{ni}^{\prime \prime \prime })^2=E_f^n(\eta _{ni}^{\prime
})^2+v_{ni}^2=E_f^n\xi _{ni}^2=I\left( f\left( t_{ni}\right) \right) .
\end{equation*}

Set $\rho _n^{*}\left( f,h\right) =\rho _n\left( f,h\right) +\rho _n^{\prime
}\left( f,h\right) ,$ where
\begin{equation*}
\rho _n^{\prime }\left( f,h\right) =\sum_{i=1}^nh\left( t_{ni}\right) \left(
\xi _{ni}-\xi _{ni}^{*}\right) .
\end{equation*}
The  lemma will be proved if we show (\ref{c-b-2-1}).
Because of the assumption (C2), it suffices  to prove that
\begin{equation}
\sup_{f\in \mathcal{F}}\sup_{h\in \mathcal{F}_f(r_n)}P_f^n\left( \left| \rho
_n^{\prime }\left( f,h\right) \right| \geq \frac \epsilon 2r_n^\alpha
\right) =O\left( r_n^{2\alpha }\right) ,  \label{c-b-3}
\end{equation}
for some $\alpha \in \left( \frac 1{2\beta },1\right) .$ To prove (\ref
{c-b-3}) we note that, by (\ref{c-b-p1}) and (\ref{c-b-p3}) we have $\xi
_{ni}-\xi _{ni}^{*}=\eta _{ni}^{\prime \prime }-\eta _{ni}^{\prime \prime
\prime }$ and therefore $\rho _n^{\prime }\left( f,h\right) $ can be
represented as follows:
\begin{equation*}
\rho _n^{\prime }\left( f,h\right) =\sum_{i=1}^nh\left( t_{ni}\right) \eta
_{ni}^{\prime \prime }-\sum_{i=1}^nh\left( t_{ni}\right) \eta _{ni}^{\prime
\prime \prime }.
\end{equation*}
>From the last equality we get
\begin{equation}
P_f^n\left( \left| \rho _n^{\prime }\left( f,h\right) \right| \geq \frac
\epsilon 2r_n^\alpha \right) \leq J_n^{\left( 1\right) }+J_n^{\left(
2\right) },  \label{c-b-5}
\end{equation}
where
\begin{equation*}
J_n^{\left( 1\right) }=P_f^n\left( \left| \sum_{i=1}^nh\left( t_{ni}\right)
\eta _{ni}^{\prime \prime }\right| \geq \frac \epsilon 4r_n^\alpha \right)
,\quad J_n^{\left( 2\right) }=P_f^n\left( \left| \sum_{i=1}^nh\left(
t_{ni}\right) \eta _{ni}^{\prime \prime \prime }\right| \geq \frac \epsilon
4r_n^\alpha \right) .
\end{equation*}

By Chebyshev's inequality we have
\begin{equation*}
J_n^{\left( 1\right) }\leq cr_n^{-2\alpha }\sum_{i=1}^nh\left( t_{ni}\right)
^2E_f^n(\eta _{ni}^{\prime \prime })^2.
\end{equation*}
Since $E_f^n(\eta _{ni}^{\prime \prime })^2=E_f^n(\xi _{ni}^{\prime \prime
}-E_f^n\xi _{ni}^{\prime \prime })^2\leq E_f^n(\xi _{ni}^{\prime \prime })^2$
and $\left\| h\right\| _\infty \leq r_n,$ making use of (\ref{c-b-p2}) and
of the strengthened version of the Lindeberg condition (C3), we obtain
\begin{equation}
J^{\left( 1\right) }\leq c\sum_{i=1}^nE_f^n(r_n^{1-\alpha }\xi _{ni})^2%
\mathbf{1}\left( \left| r_n^{1-\alpha }\xi _{ni}\right| >1\right) =O\left(
r_n^{2\alpha }\right) .  \label{c-b-6}
\end{equation}

To handle the term $J_n^{\left( 2\right) }$ on the right-hand side of (\ref
{c-b-5}) we again invoke the Chebyshev inequality to obtain
\begin{equation}
J_n^{\left( 2\right) }\leq cr_n^{-2\alpha }\sum_{i=1}^nh\left( t_{ni}\right)
^2E_f^n(\eta _{ni}^{\prime \prime \prime })^2=cr_n^{-2\alpha
}\sum_{i=1}^nh\left( t_{ni}\right) ^2v_{ni}^2.  \label{c-b-7}
\end{equation}
Since $\xi _{ni}^2=(\xi _{ni}^{\prime })^2+(\xi _{ni}^{\prime \prime })^2$
and $E^{n}_{f}\xi _{ni}^{\prime }=-E^{n}_{f}\xi _{ni}^{\prime \prime },$ we
have
\begin{equation*}
v_{ni}^2=E_f^n\xi _{ni}^2-E_f^n(\xi _{ni}^{\prime })^2 +(E^{n}_{f}\xi
_{ni}^{\prime})^2= E_f^n(\xi _{ni}^{\prime \prime })^2 +(E^{n}_{f}\xi
_{ni}^{\prime \prime })^2\leq 2E_f^n(\xi _{ni}^{\prime \prime })^2,
\end{equation*}
which in turn implies, in the same manner as for $J_n^{\left( 1\right) },$%
\begin{equation}
J_n^{\left( 2\right) }\leq cr_n^{-2\alpha }\sum_{i=1}^nh\left( t_{ni}\right)
^2E_f^n(\xi _{ni}^{\prime \prime })^2=O\left( r_n^{2\alpha }\right) .
\label{c-b-8}
\end{equation}
Inserting (\ref{c-b-6}) and (\ref{c-b-8}) into (\ref{c-b-5}) we obtain (\ref
{c-b-3}).
\end{proof}

\subsection{Construction of the likelihoods on the same probability space.}

We proceed to construct the local experiment $\mathcal{E}_{f}^{n}$ on the
same measurable space with a Gaussian experiment. For this let $\varepsilon
_{1},\varepsilon _{2}...$ be an infinite sequence of i.i.d. standard normal
r.v.'s [defined on some probability space $\left( \Omega ^{0},\mathcal{A}%
^{0},\mathbf{P}\right) $]. Consider the finite sequence of Gaussian
observations
\begin{equation}
Y_{i}^{n}=h\left( t_{ni}\right) +I^{-1/2}\left( f\left( t_{ni}\right)
\right) \varepsilon _{i},\quad i=1,...,n,  \label{c-d-1}
\end{equation}
with $f\in \mathcal{F},$ $h\in \mathcal{F}_{f}(r_{n}).$ The statistical
experiment generated by these is
\begin{equation*}
\mathcal{G}_{f}^{n}=\left( \mathrm{R}^{n},\mathcal{B}^{n},\left\{
Q_{f,h}^{n}:h\in \mathcal{F}_{f}(r_{n})\right\} \right)
\end{equation*}
and the likelihood $L_{f,h}^{0,n}=dQ_{f,h}^{n}/dQ_{f,0}^{n}$ as a r.v. under
$Q_{f,0}^{n}$ has a representation
\begin{equation}
\widetilde{L}_{f,h}^{0,n}=\exp \left( \sum_{i=1}^{n}h\left( t_{ni}\right)
\zeta _{ni}-\frac{1}{2}\sum_{i=1}^{n}h^{2}\left( t_{ni}\right) I\left(
f\left( t_{ni}\right) \right) \right) .
\end{equation}
where $\zeta _{ni}=I^{1/2}\left( f\left( t_{ni}\right) \right) \varepsilon
_{i},$ $i=1,...,n.$ It is clear that $\zeta _{n1},...,\zeta _{nn}$ is a
sequence of independent normal r.v.'s of means $0$ and variances $I\left(
f\left( t_{ni}\right) \right) ,$ on the probability space $\left( \Omega
^{0},\mathcal{A}^{0},\mathbf{P}\right) .$

We shall construct a version of the likelihoods
\begin{equation*}
L_{f,h}^{1,n}=\frac{dP_{f+h}^{n}}{dP_{f}^{n}},\quad h\in \mathcal{F}%
_{f}(r_{n})
\end{equation*}
of the experiment $\mathcal{E}_{f}^{n}$ on the probability space $\left(
\Omega ^{0},\mathcal{A}^{0},\mathbf{P}\right) $, obtaining thus an
equivalent experiment $\widetilde{\mathcal{E}}_{f}^{n}.$ For this we apply
Theorem \ref{THEOREM-HUNG-CONSTR} with $X_{ni}=\xi _{ni}^{\ast },$ $%
N_{ni}=\zeta _{ni},$ $\lambda _{n}=r_{n}^{1-1/(2{\beta })}$ and $x=\lambda
_{n}^{-1}\log n.$ According to this theorem, there is a sequence of
independent r.v.'s $\widetilde{\xi }_{ni},$ $i=1,...,n$ on the probability
space $\left( \Omega ^{0},\mathcal{A}^{0},\mathbf{P}\right) ,$ satisfying $%
\widetilde{\xi }_{ni}\overset{d}{=}\xi _{ni}^{\ast },$ for $i=1,...,n,$ and
such that
\begin{equation}
\sup_{h}\mathbf{P}\left( \left| \sum_{i=1}^{n}h\left( t_{ni}\right) \left(
\widetilde{\xi }_{ni}-\zeta _{ni}\right) \right| \geq c_{1}r_{n}^{1/\left( 2{%
\beta }\right) }\left( \log n\right) ^{2}\right) \leq c_{2}\frac{1}{n},
\label{c-d-hung}
\end{equation}
where $c_{1},$ $c_{2}$ are absolute constants and the $\sup $ is taken over $%
h\in \mathcal{F}_{f}(r_{n}).$ We recall that $r_{n}^{-1}h\in \mathcal{H}%
(\beta ,L)\subset \mathcal{H}(1/2,L).$

In order to construct the log-likelihoods $\log L_{f,h}^{1,n}$ of the
experiment $\mathcal{E}_{f}^{n}$ it suffices to construct a new
''remainder'' $\widetilde{\rho }_{n}=\widetilde{\rho }_{n}\left( f,h\right) $
on the probability space $\left( \Omega ^{0},\mathcal{A}^{0},\mathbf{P}%
\right) ,$ or on some extension of it, such that the joint distribution of
the sequence $(\widetilde{\xi }_{n1},...,\widetilde{\xi }_{nn},\widetilde{%
\rho }_{n})$ is the same as that of the original sequence $(\xi _{n1}^{\ast
},...,\xi _{nn}^{\ast },\rho _{n}^{\ast }).$ This can be done using any kind
of constructions, since the only property required from the r.v. $\widetilde{%
\rho }_{n}$ is to satisfy (\ref{c-b-2-1}), with $\widetilde{\rho }_{n}$
replacing $\rho _{n}^{\ast },$ which follows obviously from the fact that
(by construction) $\mathcal{L}(\widetilde{\rho }_{n})=\mathcal{L}(\rho
_{n}^{\ast }).$ We shall describe such a possible construction using some
elementary arguments by enlarging the initial probability space, although it
is possible to give a more delicate one on the same probability space. Let
us consider the probability space $\mathbb{S}^{\ast }\mathbb{=}\left( \Omega
^{0},\mathcal{A}^{0},\mathbf{P}\right) \otimes \left( \mathrm{R},\mathcal{B},%
\mathbf{P}_{\rho _{n}^{\ast }|\xi _{n1}^{\ast },...,\xi _{nn}^{\ast
}}\right) $ as an enlargement of the initial space $\left( \Omega ^{0},%
\mathcal{A}^{0},\mathbf{P}\right) ,$ where $\mathbf{P}_{\rho _{n}^{\ast
}|\xi _{n1}^{\ast },...,\xi _{nn}^{\ast }}$ is the conditional distribution
of $\rho _{n}^{\ast }$ given $\xi _{n1}^{\ast },...,\xi _{nn}^{\ast }.$ Now,
on the enlarged probability space $\mathbb{S}^{\ast }$ we define the r.v. $%
\widetilde{\rho }_{n}\left( \widetilde{\omega }\right) =y,$ for all $%
\widetilde{\omega }=(x_{1},...,x_{n},y)\in \mathbb{S}^{\ast },$ which has
the desired properties. Without any loss of generality we can assume that
the construction is performed on the initial probability space $\left(
\Omega ^{0},\mathcal{A}^{0},\mathbf{P}\right) .$ For more complicated
constructions we refer to Berkes and Philipp \cite{BerPh}. In any case, the
construction is performed in such a way that the new remainder $\widetilde{%
\rho }_{n}=\widetilde{\rho }_{n}\left( f,h\right) $ satisfies
\begin{equation}
\sup_{f\in \mathcal{F}}\sup_{h\in \mathcal{F}_{f}(r_{n})}\mathbf{P}\left(
\left| \widetilde{\rho }_{n}\left( f,h\right) \right| \geq 3r_{n}^{\alpha
}\right) =O\left( r_{n}^{2\alpha }\right) .  \label{c-d-1-a}
\end{equation}
Define the r.v.'s $\widetilde{L}_{f}^{n}\left( h\right) $ such that, for any
$h\in \mathcal{F}_{f}(r_{n}),$
\begin{equation}
\log \widetilde{L}_{f}^{n}\left( h\right) =\sum_{i=1}^{n}h\left(
t_{ni}\right) \widetilde{\xi }_{ni}-\frac{1}{2}\sum_{i=1}^{n}h\left(
t_{ni}\right) ^{2}I\left( f\left( t_{ni}\right) \right) +\widetilde{\rho }%
_{n}\left( f,h\right) .  \label{c-d-2}
\end{equation}
On the measurable space $\left( \Omega ^{0},\mathcal{A}^{0}\right) $
consider the set of laws $\{\widetilde{P}_{f,h}^{n}:h\in \mathcal{F}%
_{f}(r_{n})\},$ where $\widetilde{P}_{f,0}^{n}=\mathbf{P}$ and $\widetilde{P}%
_{f,h}^{n}$ is such that
\begin{equation*}
\frac{d\widetilde{P}_{f,h}^{n}}{d\widetilde{P}_{f,0}^{n}}=\widetilde{L}%
_{f,h}^{1,n},
\end{equation*}
for any $h\in \mathcal{F}_{f}(r_{n}).$ Set
\begin{equation*}
\widetilde{\mathcal{E}}_{f}^{n}=\left( \Omega ^{0},\mathcal{A}^{0},\{%
\widetilde{P}_{f,h}^{n}:h\in \mathcal{F}_{f}(r_{n})\}\right) .
\end{equation*}
Since the quadratic terms in the expansions (\ref{c-b-2}) and (\ref{c-d-2})
are deterministic, the equality in distribution of the two vectors $(\xi
_{n1}^{\ast },...,\xi _{ni}^{\ast },\rho _{n}^{\ast })$ and $(\widetilde{\xi
}_{n1},...,\widetilde{\xi }_{ni},\widetilde{\rho }_{n})$ implies for any
finite set $S\subset \mathcal{F}_{f}(r_{n})$
\begin{equation}
\mathcal{L}\left( (L_{f,h}^{1,n})_{h\in S}|P_{f}^{n}\right) =\mathcal{L}%
\left( (\widetilde{L}_{f,h}^{1,n})_{h\in S}|\widetilde{P}_{f,0}^{n}\right) .
\label{c-d-3}
\end{equation}
>From (\ref{c-d-3}) it follows that, for any $n=1,2,...$ , the experiments $%
\mathcal{E}_{f}^{n}$ and $\widetilde{\mathcal{E}}_{f}^{n}$ are (exactly)
equivalent, i.e. $\Delta \left( \mathcal{E}_{f}^{n},\widetilde{\mathcal{E}}%
_{f}^{n}\right) =0$. From the likelihood process $\widetilde{L}_{f,h}^{0,n}$%
, $h\in \mathcal{F}_{f}(r_{n})$ defined on $\left( \Omega ^{0},\mathcal{A}%
^{0},\mathbf{P}\right) $ (cf. (\ref{c-d-1-a})) we construct an equivalent
version
\begin{equation*}
\widetilde{\mathcal{G}}_{f}^{n}=\left( \Omega ^{0},\mathcal{A}^{0},\left\{
\widetilde{Q}_{f,h}^{n}:h\in \mathcal{F}_{f}(r_{n})\right\} \right)
\end{equation*}
of $\mathcal{G}_{f}^{n}$ in the same way.

\subsection{Proof of Theorem \ref{Theorem LAQ-1}}

To prove Theorem \ref{Theorem LAQ-1} we only have to verify the assumptions
of Theorem \ref{Theorem-G-1}. In our next lemma it is shown that condition (%
\ref{CC1}) is met.

\begin{lemma}
\label{Lemma HD-1}Assume that the sequence of experiments $\mathcal{E}^{n}$
satisfies condition LASE and that conditions (C1-C4) hold true. Then the
constructed experiments $\widetilde{\mathcal{E}}_{f}^{n}$ and $\widetilde{%
\mathcal{G}}_{f}^{n}$ are such that for some $\alpha \in (\frac{1}{2{\beta }}%
,1),$
\begin{equation*}
\sup_{f\in \mathcal{F}}\sup_{h\in \mathcal{F}_{f}\left( r_{n}\right) }%
\mathbf{P}\left( \left| \log \frac{d\widetilde{P}_{f,h}^{n}}{d\widetilde{P}%
_{f,0}^{n}}-\log \frac{d\widetilde{Q}_{f,h}^{n}}{d\widetilde{Q}_{f,0}^{n}}%
\right| >r_{n}^{\alpha }\right) =O\left( r_{n}^{2\alpha }\right) .
\end{equation*}
\end{lemma}

\begin{proof}
The proof is based on inequality (\ref{c-d-hung}) and of the bound (\ref
{c-b-2-1}) in Lemma \ref{Lemma-modif-exp}. Being elementary, it is left to
the reader.
\end{proof}

Next we need to check condition (\ref{CC2}) in Theorem \ref{Theorem-G-1}.

\begin{lemma}
\label{Lemma moder-P}Assume that the sequence of experiments $\mathcal{E}%
^{n} $ satisfies condition LASE and that conditions (C1-C4) hold true. Then
there is a constant $\alpha \in \left( 1/2\beta ,1\right) $ such that, for
any $\varepsilon \in (0,1),$
\begin{equation*}
\sup_{f\in \mathcal{F}}\sup_{h\in \mathcal{F}_{f}\left( r_{n}\right)
}P_{f+h}^{n}\left( \log \frac{dP_{f+h}^{n}}{dP_{f}^{n}}>-\varepsilon \log
r_{n}\right) =O\left( r_{n}^{2\alpha }\right) ,\quad n\rightarrow \infty .
\end{equation*}
\end{lemma}

\begin{proof}
Consider the inverse likelihood ratio $dP_f^n/dP_{f+h}^n$ corresponding to
the local experiment $\mathcal{E}_f^n.$ Setting $g=f+h\in \mathcal{F}$ and
using Lemma \ref{Lemma-modif-exp}, we rewrite it as
\begin{equation}
\log \frac{dP_f^n}{dP_{f+h}^n}=\log \frac{dP_{g-h}^n}{dP_g^n}%
=-\sum_{i=1}^nh\left( t_{ni}\right) \xi _{ni}^{*}\left( g\right) -\frac
12\sum_{i=1}^nh\left( t_{ni}\right) ^2I\left( g\left( t_{ni}\right) \right)
+\rho _n^{*}\left( g,h\right) ,
\end{equation}
where $h\in \mathcal{F}_g\left( r_n\right) $ and $\xi _{ni}^{*}\left(
g\right) ,$ $i=1,...,n$ are $P_g^n$-independent r.v.'s of means $0$ and
variances $E_{P_g^n}\xi _{ni}^{*}\left( g\right) ^2=I\left( g\left(
t_{ni}\right) \right) \leq I_{\max },$ $i=1,...,n.$ Moreover $\left|
r_n^{1-\alpha }\xi _{ni}^{*}\left( f\right) \right| \leq c,\quad i=1,...,n.$
Because of conditions (C1) and (C4), we have
\begin{equation}
\sum_{i=1}^nh\left( t_{ni}\right) ^2I(g(t_{ni}))=O\left( nr_n^2\right)
=O\left( 1\right) .  \label{mm-0-1}
\end{equation}
Choose $\alpha \in \left( \frac 1{2\beta },1\right) $ such that conditions
(C2-C3) hold true. With these notations, for $n$ large enough,
\begin{equation*}
P_g^n\left( \log \frac{dP_{g-h}^n}{dP_g^n}\leq \varepsilon \log r_n\right)
\leq J_n^{\left( 1\right) }+J_n^{\left( 2\right) },
\end{equation*}
where
\begin{eqnarray}
J_n^{\left( 1\right) } &=&P_g^n\left( -\sum_{i=1}^nh\left( t_{ni}\right) \xi
_{ni}^{*}\left( g\right) <\frac \varepsilon 2\log r_n\right) ,  \notag \\
J_n^{\left( 2\right) } &=&P_g^n\left( \rho _n\left( g,h\right) \geq
r_n^\alpha \right) .  \label{mm-1}
\end{eqnarray}
Since $\left\| h\right\| _\infty \leq r_n,$ it follows that the r.v.'s $h\left(
t_{ni}\right) \xi _{ni}^{*}\left( g\right) $ are bounded by $r_n^\alpha \leq
c_1,$ for some absolute constant $c_1.$ By Lemma \ref{Lemma APX-1} (see the
Appendix),
\begin{equation*}
E_f^n\exp \left( -\frac 2\varepsilon h\left( t_{ni}\right) \xi
_{ni}^{*}\left( g\right) \right) \leq \exp \left( c_2h\left( t_{ni}\right)
^2E_f^n\xi _{ni}^{*}\left( g\right) ^2\right) ,
\end{equation*}
for another absolute constant $c_2.$ Simple calculations yield
\begin{eqnarray*}
J_n^{\left( 1\right) } &\leq &e^{-2\log r_n}\prod_{i=1}^nE_g^n\exp \left(
4\varepsilon ^{-1}h\left( t_{ni}\right) \xi _{ni}^{*}\left( g\right) \right)
\\
&\leq &r_n^{-2}\exp \left( c_2\varepsilon ^{-2}\sum_{i=1}^nh\left(
t_{ni}\right) ^2E_g^n\xi _{ni}^{*}\left( g\right) ^2\right) .
\end{eqnarray*}
Taking into account $E_g^n\xi _{ni}^{*}\left( g\right) ^2=I\left( g\left(
t_{ni}\right) \right) $ and (\ref{mm-0-1}) we get $\sup J_n^{\left( 1\right)
}=O\left( r_n^2\right) =O\left( r_n^{2\alpha }\right) ,$ where the supremum is
taken over $f\in \mathcal{F}$ and $h\in \mathcal{F}_f\left( r_n\right) .$
The bound $\sup J_n^{\left( 2\right) }=O\left( r_n^{2\alpha }\right) ,$ with the
supremum over the same $f$ and $h,$ is straightforward, by assumption (C2).
Combining  the bounds for $J_n^{\left( 1\right) }$ and $J_n^{\left( 2\right)
} $ we obtain the lemma.
\end{proof}

Obviously a similar result holds true for the constructed experiment $%
\widetilde{\mathcal{E}}^{n}.$

\begin{remark}
Assume that the sequence of experiments $\mathcal{E}^{n}$ satisfies
condition LASE and that conditions (C1-C4) hold true. Then there is a
constant $\alpha \in \left( 1/2\beta ,1\right) $ such that, for any $%
\varepsilon \in (0,1),$
\begin{equation*}
\sup_{f\in \mathcal{F}}\sup_{h\in \mathcal{F}_{f}\left( r_{n}\right) }%
\widetilde{P}_{f,h}^{n}\left( \log \frac{d\widetilde{P}_{f,h}^{n}}{d%
\widetilde{P}_{f,0}^{n}}\geq -\varepsilon \log r_{n}\right) =O\left(
r_{n}^{2\alpha }\right) ,\quad n\rightarrow \infty .
\end{equation*}
\end{remark}

We continue with a moderate deviations bound for the log-likelihood ratio of
the Gaussian experiment $\mathcal{G}^n$ required by the condition (\ref{CC3}%
) of Theorem \ref{Theorem-G-1}, which is proved in the same way as the above
Lemma \ref{Lemma moder-P}.

\begin{lemma}
\label{Lemma moder-Q}Assume that the sequence of experiments $\mathcal{E}%
^{n} $ satisfies condition LASE and that conditions (C1-C4) hold true. Then
there is a constant $\alpha \in \left( 1/2\beta ,1\right) $ such that, for
any $\varepsilon \in (0,1),$%
\begin{equation*}
\sup_{f\in \mathcal{F}}\sup_{h\in \mathcal{F}_{f}\left( r_{n}\right)
}Q_{f,h}^{n}\left( \log \frac{dQ_{f,h}^{n}}{dQ_{f,0}^{n}}\geq -\varepsilon
\log r_{n}\right) =O\left( r_{n}^{2\alpha }\right) ,\quad n\rightarrow
\infty .
\end{equation*}
\end{lemma}

\begin{proof}
Consider the likelihood ratio corresponding to the local Gaussian experiment
$\mathcal{G}_f^n:$%
\begin{equation*}
L_{f,h}^{2,n}=\frac{dQ_{f+h}^n}{dQ_f^n}=\exp \left( \sum_{i=1}^nh\left(
t_{ni}\right) \zeta _{ni}-\frac 12\sum_{i=1}^nh\left( t_{ni}\right)
^2I(f(t_{ni}))\right) ,
\end{equation*}
where $f\in \mathcal{F},$ $h\in \mathcal{F}_f\left( r_n\right) $ and $\zeta
_{ni},$ $i=1,...,n$ are independent normal r.v.'s of means $0$ and variances
$I\left( f\left( t_{ni}\right) \right) \leq I_{\max }$ respectively. Then,
by Chebyshev's inequality,
\begin{equation*}
Q_f^n\left( \log L_{f,h}^{2,n}\geq -\varepsilon \log r_n\right) \leq
r_n^2E_{Q_{f,0}^n}\exp \left( 2\varepsilon ^{-1}\sum_{i=1}^nh\left(
t_{ni}\right) \zeta _{ni}-\varepsilon ^{-1}\sum_{i=1}^nh\left( t_{ni}\right)
^2I(f(t_{ni}))\right) .
\end{equation*}
Since $\left\| h\right\| _\infty \leq cn^{-1/2}$ and $\zeta _{ni},$ $%
i=1,...,n$ are independent normal r.v.'s, we get
\begin{equation*}
Q_f^n\left( \log L_{f,h}^{2,n}\geq 2\varepsilon \log r_n\right) =O\left(
r_n^{-2}\right) ,
\end{equation*}
uniformly in $f\in \mathcal{F}$ and $h\in \mathcal{F}_f\left( r_n\right) .$
\end{proof}

\section{Application to nonparametrically driven models\label{SECTION-Appl}}

We consider a particular case of the general setting of Section \ref
{Sect-gen-appr}. Assume that $\mathcal{F}$ is given by $\mathcal{F}=\Sigma
^{\beta }$ $=\mathcal{H}(\beta ,L)\cap \Theta ^{T},$ where $T=[0,1],$ and $%
\mathcal{H}(\beta ,L)$ is a H\"{o}lder ball on $T.$ Consider the case where
the experiment $\mathcal{E}^{n}$ (appearing in Section \ref{Sect-gen-appr}
in a general form) is generated by a sequence of independent observations $%
X_{1},...,X_{n}$ where each r.v. $X_{i}$ has density $p(x,f(t_{ni})),$ $f\in
\mathcal{F},$ $t_{ni}=i/n.$ The local experiment at $f\in \mathcal{F}$ then
is
\begin{eqnarray*}
\mathcal{E}_{f}^{n} &=&(X^{n},\mathcal{X}^{n},\{P_{f,h}^{n}:h\in \mathcal{F}%
_{f}(r_{n})\}), \\
P_{f,h}^{n} &=&P_{f(t_{n1})+h(t_{n1})}\times ...\times
P_{f(t_{nn})+h(t_{nn})}
\end{eqnarray*}
and $P_{\theta }$ is the distribution on $(X,\mathcal{X})$ corresponding to
the density $p(x,\theta ),$ $\theta \in \Theta .$ Let $I(\theta )$ be the
Fisher information corresponding to the density $p(x,\theta ),$ as defined
by (\ref{Fisher-inf}).

\begin{theorem}
\label{THEOREM-APPL-1}Assume that the density $p(x,\theta )$ satisfies
conditions (R1-R3). Then the sequence of experiments $\mathcal{E}^{n},$ $%
n\geq 1$ for $\mathcal{F}=\Sigma ^{\beta }$ satisfies LASE and conditions
(C1-C4) hold true with rate $r_{n}=cn^{-1/2},$ local Fisher information $%
I(\theta )$ and $\xi _{ni}(f)=\overset{\bullet }{l}(X_{i},f(t_{ni})),$ $%
i=1,...,n,$ where $\overset{\bullet }{l}(x,\theta )$ is the tangent vector
defined by (\ref{tan-vect}).
\end{theorem}

The remainder of section \ref{SECTION-Appl} will be devoted to the proof of
this theorem.

\subsection{Stochastic expansion for the likelihood ratio}

The following preliminary stochastic expansion for the likelihood ratio will
lead up to property LASE.

\begin{proposition}
\label{PROPOS-SE}Assume that the density $p(x,\theta )$ satisfies conditions
(R1-R3). Then, for any $f\in \mathcal{F}=\Sigma ^{\beta }$ and $h\in
\mathcal{F}_{f}(n^{-1/2}),$%
\begin{equation*}
\log \frac{dP_{f+h}^{n}}{dP_{f}^{n}}=2X_{n}(f,h)-4V_{n}(f,h)+\rho _{n}(f,h),
\end{equation*}
where
\begin{eqnarray*}
X_{n}(f,h) &=&\sum_{i=1}^{n}\left\{ \left( \sqrt{z_{ni}}-1\right)
-E_{f}^{n}\left( \sqrt{z_{ni}}-1\right) \right\} , \\
V_{n}(f,h) &=&\frac{1}{2}\sum_{i=1}^{n}E_{f}^{n}\left( \sqrt{z_{ni}}%
-1\right) ^{2}
\end{eqnarray*}
and
\begin{equation*}
z_{ni}=\frac{p(X_{i},f(t_{ni})+h(t_{ni}))}{p(X_{i},f(t_{ni}))}.
\end{equation*}
Moreover, there is an $\alpha \in (1/2\beta ,1),$ such that the remainder $%
\rho _{n}(f,h)$ satisfies
\begin{equation*}
\sup_{f\in \mathcal{F}}\sup_{h\in \mathcal{F}_{f}(n^{-1/2})}P_{f}^{n}\left(
\left| \rho _{n}(f,h)\right| >n^{-\alpha /2}\right) =O\left( n^{-\alpha
}\right) .
\end{equation*}
\end{proposition}

\begin{proof}
It is easy to see that
\begin{equation*}
\log \frac{dP_{f+h}^n}{dP_f^n}=\log \prod_{i=1}^nz_{ni}=\sum_{i=1}^n\log
\left( 1+\left( \sqrt{z_{ni}}-1\right) \right) ,
\end{equation*}
where $z_{ni}$ is defined in Proposition \ref{PROPOS-SE}. Note that, in view
of the equalities
\begin{equation*}
2(\sqrt{x}-1)=x-1-(\sqrt{x}-1)^2, \hspace{1cm}  \text{E}_f^nz_{ni}=1
\end{equation*}
we have
\begin{equation*}
 2\text{E}_f^n(\sqrt{z_{ni}}-1)=-\text{E}_f^n(\sqrt{z_{ni}}-1)^2.
\end{equation*}
 By elementary transformations we obtain
\begin{equation*}
\log \frac{dP_{f+n^{-1/2}h}^n}{dP_f^n}=X_n(f,h)-Y_n(f,h)-4V_n(f,h)+\Psi
_n(f,h),
\end{equation*}
where $X_n(f,h),$ $V_n(f,h)$ are defined in Proposition \ref{PROPOS-SE} and
\begin{equation*}
Y_n(f,h)=\sum_{i=1}^n\left\{ \left( \sqrt{z_{ni}}-1\right) ^2-E_f^n\left(
\sqrt{z_{ni}}-1\right) ^2\right\} ,
\end{equation*}
\begin{equation*}
\Psi _n(f,h)=\sum_{i=1}^n\left\{ \log \left( 1+\left( \sqrt{z_{ni}}-1\right)
\right) -2\left( \sqrt{z_{ni}}-1\right) -\left( \sqrt{z_{ni}}-1\right)
^2\right\} .
\end{equation*}
Now the result  follows from  Lemmas \ref{LEMMA-SE-1} and \ref
{LEMMA-SE-2}.
\end{proof}

\begin{lemma}
\label{LEMMA-SE-1}Assume that condition (R2) holds true. Then there is an $%
\alpha \in (1/2\beta ,1)$ such that
\begin{equation*}
\sup_{f\in \mathcal{F}}\sup_{h\in \mathcal{F}_{f}(n^{-1/2})}P_{f}^{n}\left(
\left| Y_{n}(f,h)\right| >n^{-\alpha /2}\right) =O\left( n^{-\alpha }\right)
.
\end{equation*}
\end{lemma}

\begin{proof}
Let $\delta \in (\frac{2\beta +1}{2\beta -1},\infty )$ be the real number
for which condition (R2) holds true (recall that $\beta \geq \frac 12$). Let
$\alpha =\frac{\delta -1}{\delta +1},$ which clearly is in the interval $%
(\frac 1{2\beta },1).$ Then there is an $\alpha ^{\prime }\in (\frac
1{2\beta },1)$ such that $\alpha ^{\prime }<\alpha .$ Set for $i=1,...,n,$%
\begin{equation}
\xi _{ni}=\sqrt{z_{ni}}-1,\quad \xi _{ni}^{\prime }=\xi _{ni}^21(\left| \xi
_{ni}\right| \leq n^{-\alpha /2}),\quad \xi _{ni}^{\prime \prime }=\xi
_{ni}^21(\left| \xi _{ni}\right| >n^{-\alpha /2})  \label{PP-1}
\end{equation}
and
\begin{equation}
\eta _{ni}^{\prime }=\xi _{ni}^{\prime }-E_f^n\xi _{ni}^{\prime },\quad \eta
_{ni}^{\prime \prime }=\xi _{ni}^{\prime \prime }-E_f^n\xi _{ni}^{\prime
\prime }.  \label{PP-2}
\end{equation}
With the above notations, we write $Y_n$ as follows:
\begin{equation}
Y_n=\sum_{i=1}^n\eta _{ni}^{\prime }+\sum_{i=1}^n\eta _{ni}^{\prime \prime }.
\label{PP-3}
\end{equation}
For the first term on the right-hand side of (\ref{PP-3}) we have
\begin{equation}
I_1\equiv P_f^n\left( \sum_{i=1}^n\eta _{ni}^{\prime }>\frac 12n^{-\alpha
^{\prime }/2}\right) \leq \exp \left( -\frac 12n^{(\alpha -\alpha ^{\prime
})/2}\right) \prod_{i=1}^nE_f^n\exp \left( n^{\alpha /2}\eta _{ni}^{\prime
}\right) ,  \label{PP-3-1}
\end{equation}
where the r.v.'s $n^{\alpha /2}\eta _{ni}$ are bounded by $2n^{-\alpha
/2}\leq 2.$ According to Lemma \ref{Lemma APX-1}, with $\lambda =1,$ one
obtains
\begin{equation}
E_f^n\exp \left( n^{\alpha /2}\eta _{ni}^{\prime }\right) \leq \exp \left(
cn^\alpha E_f^n(\eta _{ni}^{\prime })^2\right) ,\quad i=1,...,n.
\label{PP-3-2}
\end{equation}
Using (\ref{PP-2}) and (\ref{PP-1}),
\begin{equation}
E_f^n(\eta _{ni}^{\prime })^2\leq 2n^{-\alpha }E_f^n(\sqrt{z_{ni}}%
-1)^2,\quad i=1,...,n.  \label{PP-3-3}
\end{equation}
Set for brevity
\begin{equation*}
\overset{\bullet }{l}_{ni}(f,h)=\overset{\bullet }{l}%
_{f(t_{ni})}(X_i,f(t_{ni})+h(t_{ni})),\quad i=1,...,n,
\end{equation*}
where $\overset{\bullet }{l}_\theta (x,u)$ is the extended tangent vector
defined by (\ref{EXTEND-TG-V}) and $f\in \mathcal{F},$ $h\in \mathcal{F}%
_f(n^{-1/2}).$ With these notations,
\begin{equation}
\sqrt{z_{ni}}-1=h(t_{ni})\overset{\bullet }{l}_{ni}(f,h),\quad i=1,...,n,
\label{PP-3-4}
\end{equation}
where $n^{1/2}h\in \mathcal{H}(\beta ,L).$ Condition (R2) and $\left\|
n^{1/2}h\right\| _\infty \leq L$ imply
\begin{equation}
E_f^n(\sqrt{z_{ni}}-1)^2\leq cn^{-1},\quad i=1,...,n.  \label{PP-3-5}
\end{equation}
Inserting  these bounds into (\ref{PP-3-3}) and then invoking  the bounds obtained in (%
\ref{PP-3-2}), we obtain
\begin{equation*}
\prod_{i=1}^nE_f^n\exp \left( n^{\alpha /2}\eta _{ni}^{\prime }\right) \leq
\exp \left( c_1\right) \leq c_2.
\end{equation*}
Then, since $\alpha >\alpha ^{\prime },$ from (\ref{PP-3-1}), we  the
estimate $I_1=O(n^{-\alpha ^{\prime }})$ follows. In the same way we establish a
bound for the lower tail probability.

Now consider  the second term on the right-hand side of (\ref{PP-3}%
). For this we note that by (\ref{PP-3-4}),
\begin{equation}
\sum_{i=1}^nE_f^n\left| \sqrt{z_{ni}}-1\right| ^{2\delta }\leq cn^{-\delta
}\sum_{i=1}^nE_f^n|\overset{\bullet }{l}_{ni}(f,h)|^{2\delta }\leq
cn^{1-\delta }.  \label{PP-5}
\end{equation}
By virtue of (\ref{PP-1}) and (\ref{PP-5}),
\begin{equation*}
\sum_{i=1}^nE_f^n\xi _{ni}^{\prime \prime }\leq cn^{\alpha (\delta
-1)}\sum_{i=1}^nE_f^n\left| \sqrt{z_{ni}}-1\right| ^{2\delta }\leq
cn^{\alpha (\delta -1)}n^{1-\delta }=cn^{-2\alpha }.
\end{equation*}
Then since $\frac 12n^{-\alpha ^{\prime }/2}-cn^{-2\alpha }$ is positive
for $n$ large enough, we get
\begin{eqnarray*}
I_2 &\equiv &P_f^n\left( \left| \sum_{i=1}^n\eta _{ni}^{\prime \prime
}\right| >\frac 12n^{-\alpha ^{\prime }/2}\right) \leq P_f^n\left(
\sum_{i=1}^n\xi _{ni}^{\prime \prime }>\frac 12n^{-\alpha ^{\prime
}/2}-cn^{-2\alpha }\right) \\
&\leq &P_f^n\left( \max_{1\leq i\leq n}\left| \sqrt{z_{ni}}-1\right|
>n^{-\alpha /2}\right) .
\end{eqnarray*}
The last probability can be bounded, using (\ref{PP-5}), in the following way:
for any absolute constant $c>0,$%
\begin{equation}
P_f^n\left( \max_{1\leq i\leq n}\left| \sqrt{z_{ni}}-1\right| >cn^{-\alpha
/2}\right) \leq c_1n^{\alpha \delta }\sum_{i=1}^nE_f^n\left| \sqrt{z_{ni}}%
-1\right| ^{2\delta }\leq c_2n^{\alpha \delta }n^{1-\delta }=c_2n^{-\alpha }.
\label{PP-6}
\end{equation}
This yields  $I_2=O(n^{-\alpha ^{\prime }}).$ The bounds for $I_1$ (with
the corresponding bound of the lower tail) and for $I_2,$ in conjunction  with (\ref
{PP-3}), obviously imply the  assertion.
\end{proof}

\begin{lemma}
\label{LEMMA-SE-2}Assume that condition (R2) holds true. Then, there is an $%
\alpha \in (1/2\beta ,1),$ such that
\begin{equation*}
\sup_{f\in \mathcal{F}}\sup_{h\in \mathcal{F}_{f}(n^{-1/2})}P\left( \left|
\Psi _{n}(f,h)\right| >n^{-\alpha /2}\right) =O\left( n^{-\alpha }\right) .
\end{equation*}
\end{lemma}

\begin{proof}
We keep the notations from Lemma \ref{LEMMA-SE-1}. Additionally set for $%
i=1,...,n,$
\begin{equation*}
\psi _{ni}=\log (1+(\sqrt{z_{ni}}-1))-2(\sqrt{z_{ni}}-1)+(\sqrt{z_{ni}}-1)^2.
\end{equation*}
Then we can represent $\Psi _n(f,h)$ as follows: $\Psi _n(f,h)=\Psi _1+\Psi
_2,$ where
\begin{equation*}
\Psi _1=\sum_{i=1}^n\psi _{ni}\mathbf{1}\left( |\sqrt{z_{ni}}-1|\leq
n^{-\alpha /2}\right) ,\quad \Psi _2=\sum_{i=1}^n\psi _{ni}\mathbf{1}\left( |%
\sqrt{z_{ni}}-1|>n^{-\alpha /2}\right) .
\end{equation*}
Assume that $n$ is large enough so that $n^{-\alpha /2}\leq 1/2.$ Then a simple
Taylor expansion gives $\left| \psi _{ni}\right| \leq c|\sqrt{%
z_{ni}}-1|^3,$ provided that $|\sqrt{z_{ni}}-1|\leq n^{-\alpha /2}.$ This in
turn implies
\begin{equation*}
\left| \Psi _1\right| \leq c\max_{1\leq i\leq n}|\sqrt{z_{ni}}%
-1|\sum_{i=1}^n(\sqrt{z_{ni}}-1)^2.
\end{equation*}
Since by (\ref{PP-3-5}) one has $E_f^n(\sqrt{z_{ni}}-1)^2\leq cn^{-1},$ we
obtain
\begin{equation*}
\left| \Psi_1\right|  \leq c_1\max_{1\leq i\leq
n}|\sqrt{z_{ni}}-1|\left( \left| Y_n\right| +c_2\right) .
\end{equation*}
Therefore
\begin{equation*}
P\left( \left| \Psi _1\right| \geq \frac 12n^{-\alpha ^{\prime }/2}\right)
\leq P\left( \left| Y_n\right| >n^{-\alpha ^{\prime }/2}\right) +P\left(
\max_{1\leq i\leq n}|\sqrt{z_{ni}}-1|>cn^{-\alpha ^{\prime }/2}\right) .
\end{equation*}
Now from Lemma \ref{LEMMA-SE-1} and (\ref{PP-6}) we obtain  the
bound
\begin{equation}
P\left( \left| \Psi _1\right| \geq n^{-\alpha ^{\prime }/2}\right)
=O(n^{-\alpha ^{\prime }}).  \label{PSI-1-1}
\end{equation}
As to $\Psi _2,$ we have
\begin{equation*}
\left\{ \left| \Psi _2\right| >\frac 12n^{-\alpha ^{\prime }/2}\right\}
\subset \left\{ \max_{1\leq i\leq n}|\sqrt{z_{ni}}-1|>cn^{-\alpha ^{\prime
}/2}\right\} ,
\end{equation*}
from which we deduce, by (\ref{PP-6}),
\begin{equation}
P\left( \left| \Psi _2\right| >\frac 12n^{-\alpha ^{\prime }/2}\right) \leq
P\left( \max_{1\leq i\leq n}|\sqrt{z_{ni}}-1|>cn^{-\alpha ^{\prime
}/2}\right) =O(n^{-\alpha ^{\prime }}).  \label{PSI-1-2}
\end{equation}
The result   follows  from (\ref{PSI-1-1}) and (\ref
{PSI-1-2}).
\end{proof}

\subsection{Proof of Theorem \ref{THEOREM-APPL-1}}

We split the proof into two lemmas, in such a way that Theorem \ref
{THEOREM-APPL-1} follows immediately from Proposition \ref{PROPOS-SE} and
these lemmas.

Set
\begin{equation*}
M_n(f,h)=\sum_{i=1}^nh(t_{ni})\overset{\bullet }{l}_{ni}(f),
\end{equation*}
where $\overset{\bullet }{l}_{ni}(f)=\overset{\bullet }{l}(X_i,f(t_{ni}))$
and $\overset{\bullet }{l}(x,\theta )$ is the tangent vector defined by (\ref
{tan-vect}).

\begin{lemma}
\label{LEMMA-SE-3}Assume that conditions (R1-R3) hold true. Then there is an
$\alpha \in (1/2\beta ,1)$ such that for $\mathcal{F}=\Sigma ^{\beta }$
\begin{equation*}
\sup_{f\in \mathcal{F}}\sup_{h\in \mathcal{F}_{f}(n^{-1/2})}P_{f}^{n}\left(
\left| 2X_{n}(f,h)-M_{n}(f,h)\right| >n^{-\alpha /2}\right) =O\left(
n^{-\alpha }\right) .
\end{equation*}
\end{lemma}

\begin{proof}
Let $\delta _1\in (\frac 1{2\beta },1)$ and $\delta _2\in (\frac{2\beta +1}{%
2\beta -1},\infty )$ be respectively the real numbers for which
conditions (R1) and (R2) hold true, where $\beta \geq \frac 12.$ Let $%
\alpha =\min \{\delta _1,\frac{\delta _2-1}{\delta _2+1}\},$ which clearly
is in the interval $(\frac 1{2\beta },1).$ Then there is an $\alpha ^{\prime
}\in (\frac 1{2\beta },1)$ such that $\alpha ^{\prime }<\alpha .$ Denote
for $i=1,...,n,$%
\begin{equation}
\xi _{ni}=2\left( \sqrt{z_{ni}}-1\right) -2E_f^n\left( \sqrt{z_{ni}}%
-1\right) -h(t_{ni})\overset{\bullet }{l}_{ni}(f),  \label{SE3-0}
\end{equation}
\begin{equation}
\xi _{ni}^{\prime }=\xi _{ni}1(\left| \xi _{ni}\right| \leq n^{-\alpha
/2}),\quad \xi _{ni}^{\prime \prime }=\xi _{ni}1(\left| \xi _{ni}\right|
>n^{-\alpha /2})  \label{SE3-1}
\end{equation}
and
\begin{equation}
\eta _{ni}^{\prime }=\xi _{ni}^{\prime }-E_f^n\xi _{ni}^{\prime },\quad \eta
_{ni}^{\prime \prime }=\xi _{ni}^{\prime \prime }-E_f^n\xi _{ni}^{\prime
\prime }.  \label{SE3-2}
\end{equation}
With these notations
\begin{equation}
2X_n(f,h)-M_n(f,h)=\sum_{i=1}^n\eta _{ni}^{\prime }+\sum_{i=1}^n\eta
_{ni}^{\prime \prime }.  \label{SE3-3}
\end{equation}
Consider  the first term on the right-hand side of (\ref{SE3-3}). Since
the r.v.'s $n^{\alpha /2}\eta _{ni}$ are bounded by $2,$ we have by Lemma
\ref{Lemma APX-1} with $\lambda =1,$%
\begin{equation*}
E_f^n\exp \left( n^{\alpha /2}\eta _{ni}^{\prime }\right) \leq \exp \left(
cn^\alpha E_f^n(\eta _{ni}^{\prime })^2\right) ,\quad i=1,...,n.
\end{equation*}
Then
\begin{equation}
I_1\equiv P_f^n\left( \sum_{i=1}^n\eta _{ni}^{\prime }>\frac 12n^{-\alpha
^{\prime }/2}\right) \leq \exp \left( -\frac 12n^{\left( \alpha -\alpha
^{\prime }\right) /2}+cn^\alpha \sum_{i=1}^nE_f^n(\eta _{ni}^{\prime
})^2\right) .  \label{SE3-4}
\end{equation}
>From assumption (R1)  we easily obtain
\begin{equation*}
\sum_{i=1}^nE_f^n(\eta _{ni}^{\prime })^2\leq 2\sum_{i=1}^nE_f^n\xi
_{ni}^2\leq cn^{-2\delta _1}\leq cn^{-2\alpha },
\end{equation*}
which in conjunction  with (\ref{SE3-4})  implies $I_1=O(n^{-\alpha ^{\prime }}).$
The bound for the lower tail can be established analogously.

For an estimate of  the second term in the right-hand side of (\ref{SE3-3}), we note
that  (\ref{SE3-0}), (\ref{SE3-1}), (\ref{SE3-2}) imply
\begin{equation*}
\sum_{i=1}^nE_f^n\left| \eta _{ni}^{\prime \prime }\right| ^2\leq cn^{\alpha
(\delta -1)}\left( \sum_{i=1}^nE_f^n\left| \sqrt{z_{ni}}-1\right| ^{2\delta
}+\frac c{n^\delta }\sum_{i=1}^nE_f^n|\overset{\bullet }{l}%
_{ni}(f)|^{2\delta }\right) .
\end{equation*}
Now assumption (R1) and (\ref{PP-5}) imply
\begin{equation*}
\sum_{i=1}^nE_f^n\left| \eta _{ni}^{\prime \prime }\right| ^{2\delta }\leq
cn^{\alpha (\delta -1)-1-\delta }=cn^{-2\alpha }.
\end{equation*}
Thus we obtain
\begin{equation*}
I_2\equiv P_f^n\left( \sum_{i=1}^n\eta _{ni}^{\prime \prime }>\frac
12n^{-\alpha ^{\prime }/2}\right) \leq cn^{\alpha ^{\prime
}}\sum_{i=1}^nE_f^n\left| \eta _{ni}^{\prime \prime }\right| ^2\leq
cn^{\alpha ^{\prime }-2\alpha }\leq cn^{-\alpha ^{\prime }}.
\end{equation*}
The bounds for $I_1$ (with the corresponding bound of the lower tail) and
for $I_2,$ in conjunction  with (\ref{SE3-3}) imply the lemma.
\end{proof}

\begin{lemma}
\label{LEMMA-SE-4}Assume that conditions (R1-R3) hold true. Then, there is
an $\alpha \in (1/2\beta ,1),$ such that for $\mathcal{F}=\Sigma ^{\beta }$
\begin{equation*}
\sup_{f\in \mathcal{F}}\sup_{h\in \mathcal{F}_{f}(n^{-1/2})}\left|
V_{n}(f,h)-\frac{1}{8}\sum_{i=1}^{n}h(t_{ni})^{2}I(f(t_{ni}))\right| \leq
n^{-\alpha /2}.
\end{equation*}
\end{lemma}

\begin{proof}
Let $\delta _1\in (\frac 1{2\beta },1)$ and $\delta _2\in (\frac{2\beta +1}{%
2\beta -1},\infty ),$ be respectively the real numbers for which
conditions (R1) and (R2) hold true, where $\beta \geq \frac 12.$ Let $%
\alpha =\min \{\delta _1,\frac{\delta _2-1}{\delta _2+1}\}.$ Set
\begin{equation*}
\xi _{ni}=\sqrt{z_{ni}-1},\quad \eta _{ni}=\frac 12h(t_{ni})\overset{\bullet
}{l}_{ni}(f),\quad i=1,...,n.
\end{equation*}
Then, for $i=1,...,n,$%
\begin{equation*}
\left| E_f^n\left( \xi _{ni}^2-\eta _{ni}^2\right) \right| ^2\leq E_f^n(\xi
_{ni}-\eta _{ni})^2E_f^n(\xi _{ni}+\eta _{ni})^2.
\end{equation*}
In view of assumption (R1), we have $\text{E}_f^n(\xi _{ni}-\eta _{ni})^2\leq
cn^{-1-\alpha }$, and  assumption (R2) implies  $\text{E}_f^n(\xi
_{ni}+\eta _{ni})^2\leq cn^{-1}.$ Thus,
\begin{equation*}
\left| E_f^n\left( \xi _{ni}^2-\eta _{ni}^2\right) \right| \leq
cn^{-1-\alpha /2},\quad i=1,...,n.
\end{equation*}
Finally
\begin{equation*}
\left| V_n(f,h)-\frac 18\sum_{i=1}^nh(t_{ni})^2I(f(t_{ni}))\right| =\frac
12\sum_{i=1}^n\left| E_f^n\left( \xi _{ni}^2-\eta _{ni}^2\right) \right|
\leq cn^{-\alpha /2}.
\end{equation*}
\end{proof}

\section{Proof of the local result\label{SECTION-Local}}

\subsection{Proof of Theorem \ref{Theorem R-2}}

Let $\beta >1/2$ and $f\in \mathcal{F}=\Sigma ^{\beta }.$ Recall that $%
\gamma _{n}$ is defined by (\ref{nonpar-rate}). Let $\alpha \in \left(
1/2\beta ,1\right) $ be the absolute constant in Theorem \ref{THEOREM-APPL-1}
and $d=\alpha -1/2\beta \leq 1.$ If we set $\alpha ^{\prime }=1/2\beta +qd$
where the absolute constant $q\in (0,1)$ will be specified later on, then $%
\alpha ^{\prime }\in \left( 1/2\beta ,1\right) $ and $\alpha ^{\prime
}<\alpha .$ Set $\delta _{n}=\gamma _{n}^{2\alpha ^{\prime }}$ and $%
M_{n}=[1/\delta _{n}]$; then clearly $M_{n}=O\left( \delta _{n}^{-1}\right)
. $ Set $t_{i}=i/n,$ $i=0,...,n.$ Let $a_{k}=\max \{t_{i}:t_{i}\leq \frac{k}{%
M_{n}}\},$ $k=0,...,M_{n}.$ Consider a partition of the unit interval $[0,1]$
into subintervals $A_{k}=(a_{k-1},a_{k}]$ where $k=1,...,M_{n}.$ To each
interval $A_{k}$ we attach the affine linear map $a_{k}(t):A_{k}\rightarrow
\lbrack 0,1]$ which transforms $A_{k}$ into the unit interval $[0,1].$ It is
clear that $\left| a_{k}(t)-a_{k}(s)\right| \leq c\delta _{n}^{-1}\left|
t-s\right| ,$ for $t,s\in A_{k}.$ Denote by $n_{k}$ the number of elements
in the set $\{i:t_{i}\in A_{k}\}$; it obviously satisfies $n\delta
_{n}=O\left( n_{k}\right) .$

Consider the local experiment $\mathcal{E}_{f}^{n}$ defined by a sequence of
independent r.v.'s $X_{1},...,X_{n},$ where each $X_{i}$ has the density $%
p(x,g(t_{i}))$ with $g=f+h,$ $h\in \Sigma _{f}^{\beta }(\gamma _{n}).$ Since
$[0,1]=\sum_{k=1}^{M_{n}}A_{k},$ we have in view of the independence of the $%
X_{i}$
\begin{equation*}
\mathcal{E}_{f}^{n}=\mathcal{E}_{f}^{n,1}\otimes ...\otimes \mathcal{E}%
_{f}^{n,M_{n}},
\end{equation*}
where the experiment $\mathcal{E}_{f}^{n,k}$ is generated by those
observations $X_{i}$ for which $t_{i}\in A_{k}.$ Set for brevity $%
f_{k}=f(a_{k}^{-1}(\cdot )),$ $g_{k}=g(a_{k}^{-1}(\cdot ))$ and $%
h_{k}=g_{k}-f_{k}.$ It is easy to see that $n_{k}^{1/2}h_{k}\in \mathcal{H}%
(\beta ,L_{1}),$ for some positive absolute constant $L_{1}.$ This means
that $h_{k}\in \mathcal{F}_{f_{k}}(n_{k}^{-1/2}).$ Consequently
\begin{equation*}
\mathcal{E}_{f}^{n,k}=\left( X^{n_{k}},\mathcal{X}^{n_{k}},\left\{
P_{f_{k}+h_{k}}^{n_{k}}:h\in \Sigma _{f}^{\beta }(\gamma _{n})\right\}
\right) ,
\end{equation*}
where $P_{s}^{n_{k}}=P_{s(1/n_{k})}\times ...\times P_{s(1)},$ for any
function $s\in \mathcal{F},$ and $P_{\theta }$ is the distribution on $(X,%
\mathcal{X})$ corresponding to the density $p(x,\theta ).$ It is clear that $%
\mathcal{E}_{f}^{n,k}$ is just a subexperiment of
\begin{equation*}
\mathcal{E}_{f_{k}}^{n_{k}}=\left( X^{n_{k}},\mathcal{X}^{n_{k}},\left\{
P_{f_{k}+h}^{n_{k}}:h\in \mathcal{F}_{f_{k}}(n_{k}^{-1/2})\right\} \right) .
\end{equation*}
Exactly in the same way we introduce the Gaussian counterparts of $\mathcal{E%
}_{f}^{n,k}:$ if $\mathcal{G}_{f}^{n}$ denotes the Gaussian experiment
introduced in Theorem \ref{Theorem R-2} then
\begin{equation*}
\mathcal{G}_{f}^{n}=\mathcal{G}_{f}^{n,1}\otimes ...\otimes \mathcal{G}%
_{f}^{n,M_{n}},
\end{equation*}
where the experiment $\mathcal{G}_{f}^{n,k}$ is generated by those
observations $Y_{i}$ (see Theorem \ref{Theorem R-2}) for which $t_{i}\in
A_{k}:$%
\begin{equation*}
\mathcal{G}_{f}^{n,k}=\left( R^{n_{k}},\mathcal{B}^{n_{k}},\left\{
Q_{f,h}^{n_{k}}:h\in \Sigma _{f}^{\beta }(\gamma _{n})\right\} \right) ,
\end{equation*}
where $Q_{s,u}^{n_{k}}=Q_{s(1/n_{k}),u(1/n_{k})}\times ...\times
Q_{s(1),u(1)},$ for any functions $s,u\in \mathcal{F},$ and $Q_{\theta ,\mu
} $ is the normal distribution with mean $\mu $ and variance $I(\theta
)^{-1}.$ It is clear that $\mathcal{G}_{f}^{n,k}$ is just a subexperiment of
\begin{equation*}
\mathcal{G}_{f_{k}}^{n_{k}}=\left( R^{n_{k}},\mathcal{B}^{n_{k}},\left\{
Q_{f_{k},h}^{n_{k}}:h\in \mathcal{F}_{f_{k}}(n_{k}^{-1/2})\right\} \right) .
\end{equation*}
According to Theorems \ref{THEOREM-APPL-1} and \ref{Theorem LAQ-1}, for any $%
f\in \mathcal{F}$ there is an experiment
\begin{equation*}
\widetilde{\mathcal{E}}_{f}^{n_{k}}=\left( \Omega ^{0},\mathcal{A}%
^{0},\left\{ \widetilde{P}_{f,h}^{n_{k}}:h\in \mathcal{F}_{f}(n_{k}^{-1/2})%
\right\} \right) ,
\end{equation*}
equivalent to $\mathcal{E}_{f_{k}}^{n_{k}}$ and an equivalent version $%
\widetilde{\mathcal{G}}_{f}^{n_{k}}$ of $\mathcal{G}_{f_{k}}^{n_{k}},$
defined on the same measurable space $(\Omega ^{0},\mathcal{A}^{0})$ with
measures $\widetilde{Q}_{f,h}^{n_{k}}$, such that uniformly in $f\in
\mathcal{F}$ and $h\in \mathcal{F}_{f}(n_{k}^{-1/2}),$%
\begin{equation}
H^{2}\left( \widetilde{P}_{f,h}^{n_{k}},\widetilde{Q}_{f,h}^{n_{k}}\right)
=O\left( n_{k}^{-\alpha }\right)  \label{LL-2}
\end{equation}
for some $\alpha \in (1/2\beta ,1).$ Set
\begin{eqnarray*}
\widetilde{\mathcal{E}}_{f}^{n,k} &=&\left( R^{n_{k}},\mathcal{B}%
^{n_{k}},\left\{ P_{f_{k},h_{k}}^{n_{k}}:h\in \Sigma _{f}^{\beta }(\gamma
_{n})\right\} \right) . \\
\widetilde{\mathcal{E}}_{f}^{n} &=&\widetilde{\mathcal{E}}_{f}^{n,1}\otimes
...\otimes \widetilde{\mathcal{E}}_{f}^{n,M_{n}}.
\end{eqnarray*}
and define $\widetilde{\mathcal{G}}_{f}^{n}$ analogously. Since $\widetilde{%
\mathcal{E}}_{f}^{n,k}$ and $\mathcal{E}_{f}^{n,k}$ are (exactly)
equivalent, it follows that
\begin{equation*}
\Delta \left( \mathcal{E}_{f}^{n},\widetilde{\mathcal{E}}_{f}^{n}\right)
=\Delta \left( \mathcal{G}_{f}^{n},\widetilde{\mathcal{G}}_{f}^{n}\right) =0
\end{equation*}
which in turn implies
\begin{equation*}
\Delta \left( \mathcal{E}_{f}^{n},\mathcal{G}_{f}^{n}\right) =\Delta \left(
\widetilde{\mathcal{E}}_{f}^{n},\widetilde{\mathcal{G}}_{f}^{n}\right) .
\end{equation*}
In view of (\ref{Le Cam-to-L1}), (\ref{L1-to-Hell}), (\ref{Hell-square}) and
(\ref{LL-2}), we have
\begin{eqnarray*}
\Delta \left( \widetilde{\mathcal{E}}_{f}^{n},\widetilde{\mathcal{G}}%
_{f}^{n}\right) &\leq &2\sum_{k=1}^{M_{n}}\sup_{h\in \Sigma _{f}^{\beta
}(\gamma _{n})}H^{2}\left( \widetilde{P}_{f_{k},h_{k}}^{n,k},\widetilde{Q}%
_{f_{k},h_{k}}^{n,k}\right) \\
&=&O\left( M_{n}n_{k}^{-\alpha }\right) =O\left( \delta _{n}^{-1}(n\delta
_{n})^{-\alpha }\right) .
\end{eqnarray*}
Choosing $q\leq \frac{1}{4},$ by an elementary calculation we obtain $\delta
_{n}^{-1}(n\delta _{n})^{-\alpha }=o\left( 1\right) .$ Thus Theorem \ref
{Theorem R-2} is proved.

\subsection{Homoscedastic form of the local result}

In order to globalize the local result in Theorem \ref{Theorem R-2}, we need
to transform the heteroscedastic Gaussian approximation into a homoscedastic
one. This is done by means of any transformation $\Gamma (\theta )$ of the
functional parameter $f$ which satisfies $\Gamma ^{\prime }\left( \theta
\right) =\sqrt{I(\theta )},$ where $I(\theta )$ is the Fisher information in
the parametric model $\mathcal{E}.$ We shall derive the following corollary
of Theorem \ref{Theorem R-2}.

\begin{corollary}
\label{CORR-VAR_ST}Let ${\beta }>1/2$ and $I\left( \theta \right) $ be the
Fisher information in the parametric experiment $\mathcal{E}.$ Assume that
the density $p(x,\theta )$ satisfies the regularity conditions (R1-R3).
Assume also that $I(\theta )$, as a function of $\theta $, satisfies a
H\"{o}lder condition with exponent $\alpha \in (1/2\beta ,1).$ For any $f\in
\Sigma ^{\beta },$ let $\mathcal{G}_{f}^{n}$ be the local Gaussian
experiment generated by observations
\begin{equation*}
Y_{i}^{n}=\Gamma (g\left( i/n\right) )+\varepsilon _{i},\quad i=1,...,n,
\end{equation*}
with $g=f+h,$ $h\in \Sigma _{f}^{\beta }\left( \gamma _{n}\right) ,$ where $%
\varepsilon _{1},...,\varepsilon _{n}$ is a sequence of i.i.d. standard
normal r.v.'s (not depending on $f$). Then, uniformly in $f\in \Sigma ,$ the
sequence of local experiments $\mathcal{E}_{f}^{n},$ $n=1,2,...$ is
asymptotically equivalent to the sequence of local Gaussian experiments $%
\mathcal{G}_{f}^{n},$ $n=1,2,...:$%
\begin{equation*}
\sup_{f\in \Sigma ^{\beta }}\Delta \left( \mathcal{E}_{f}^{n},\mathcal{G}%
_{f}^{n}\right) \rightarrow 0,\quad \text{as}\quad n\rightarrow \infty .
\end{equation*}
\end{corollary}

\begin{proof}
It will be shown that the Gaussian experiments $%
\mathcal{G}_f^{1,n}$ and $\mathcal{G}_f^{2,n}$ are asymptotically
equivalent, where $\mathcal{G}_f^{1,n}$ is generated by  observations
\begin{equation}
Y_i^n=g\left( i/n\right) +I(f(i/n))^{-1/2}\varepsilon _i,\quad i=1,...,n
\label{VV-1}
\end{equation}
and $\mathcal{G}_f^{2,n}$ is generated by observations
\begin{equation}
Y_i^n=\Gamma (g\left( i/n\right) )+\varepsilon _i,\quad i=1,...,n,
\label{VV-2}
\end{equation}
with $g=f+h,$ $h\in \Sigma _f^\beta (\gamma _n)$ and $\varepsilon
_1,...,\varepsilon _n$ being a sequence of i.i.d. standard normal r.v.'s.
Since $\Gamma ^{\prime }(\theta )=\sqrt{I(\theta )}$ and $I(\theta )$ satisfies
a H\"older condition with exponent $\alpha \in (\frac 1{2\beta },1),$ a  Taylor
expansion yields
\begin{eqnarray*}
\Gamma (\theta +u)-\Gamma (\theta ) &=&u\sqrt{I(\theta )}+u\left( \sqrt{%
I(\theta +u)}-\sqrt{I(\theta )}\right) \\
&=&u\sqrt{I(\theta )}+o(\left| u\right| ^{1+\alpha }).
\end{eqnarray*}
Then, taking into account (\ref{nonpar-rate}), we arrive at
\begin{equation*}
\Gamma (g(i/n))-\Gamma (f(i/n))=h(i/n)\sqrt{I(f(i/n))}+o(n^{-1/2}).
\end{equation*}
Set for brevity $m_i^1=\Gamma (g(i/n))-\Gamma (f(i/n))$ and $m_i^2=h(i/n)%
\sqrt{I(f(i/n))}.$ Let $Q_{f,h}^{1,n}$ and $Q_{f+h}^{2,n}$ be the
probability measures induced by (\ref{VV-1}) and (\ref{VV-2}). Then, using (%
\ref{Hell-square}) and (\ref{Hell-Gauss}), the Hellinger distance between $%
Q_{f,h}^{1,n}$ and $Q_{f+h}^{2,n}$ can easily be  seen to satisfy
\begin{equation*}
\frac 12H^2\left( Q_{f,h}^{1,n},Q_{f+h}^{2,n}\right) =1-\exp \left( -\frac
18\sum_{i=1}^n\left( m_i^1-m_i^2\right) ^2\right) =o(1),\quad n\rightarrow
\infty .
\end{equation*}
The claim on Le Cam distance convergence now follows from (\ref{Le
Cam-to-L1}) and (\ref{L1-to-Hell}).
\end{proof}

\section{Proof of the global result\label{SECTION-Global}}

In this section we shall prove Theorem \ref{Theorem R-1}.

Let ${\mathcal{E}}^{n}$ and $\mathcal{G}^{n}$ be the global experiments
defined in Theorem \ref{Theorem R-1}. Let $f\in \Sigma $ (we shall omit the
superscript $\beta $ from notation $\Sigma ^{\beta }$ and $\Sigma
_{f}^{\beta }(\gamma _{n})$) Denote by $J^{\prime }$ and $J^{\prime \prime }$
the sets of odd and even numbers, respectively, in $J=\left\{
1,...,n\right\} .$ Put
\begin{equation*}
X^{\prime ,n}=\prod_{i\in J^{\prime }}X^{\left( i\right) },\quad X^{\prime
\prime ,n}=\prod_{i\in J^{\prime \prime }}X^{\left( i\right) },\quad
R^{\prime \prime ,n}=\prod_{i\in J^{\prime \prime }}R^{(i)},\quad \mathbf{S}%
^{n}=\prod_{i=1}^{n}\mathbf{S}_{i},
\end{equation*}
where $X^{\left( i\right) }=X{,}$ $R^{\left( i\right) }=\mathrm{R}{,}$ $%
\mathbf{S}_{i}=X$ if $i$ is odd and $\mathbf{S}_{i}=\mathrm{R}$ if $i$ is
even, $i\in J.$ Consider the following product (local) experiments
corresponding to observations with even indices $i\in J:$
\begin{equation*}
{\mathcal{E}}_{f}^{\prime \prime ,n}=\bigotimes_{i\in J^{\prime \prime }}{%
\mathcal{E}}_{f}^{\left( i\right) },\quad \mathcal{G}_{f}^{\prime \prime
,n}=\bigotimes_{i\in J^{\prime \prime }}\mathcal{G}_{f}^{\left( i\right) },
\end{equation*}
where
\begin{eqnarray*}
{\mathcal{E}}_{f}^{\left( i\right) } &=&\left( {X},\mathcal{X},\left\{
P_{g(i/n)}:g=f+h,\;h\in \Sigma _{f}(\gamma _{n})\right\} \right) , \\
\mathcal{G}_{f}^{\left( i\right) } &=&\left( {\ }\mathrm{R},{\mathcal{B}}%
,\left\{ Q_{g(i/n)}:g=f+h,\;h\in \Sigma _{f}(\gamma _{n})\right\} \right) .
\end{eqnarray*}
Along with this introduce the global experiments
\begin{equation*}
{\mathcal{E}}^{\prime ,n}=\bigotimes_{i\in J^{\prime }}{\mathcal{E}}^{\left(
i\right) },\quad {\mathcal{F}}^{n}=\bigotimes_{i=1}^{n}{\mathcal{F}}^{\left(
i\right) },
\end{equation*}
where,
\begin{equation*}
\mathcal{F}^{\left( i\right) }=\left\{
\begin{array}{l}
\mathcal{E}^{\left( i\right) },\quad \text{if}\quad i\quad \text{is odd,} \\
\mathcal{G}^{\left( i\right) },\quad \text{if}\quad i\quad \text{is even,}
\end{array}
\right.
\end{equation*}
and
\begin{equation*}
{\mathcal{E}}^{\left( i\right) }=\left( X,\mathcal{X},\left\{
P_{f(t_{i})}:f\in \Sigma \right\} \right) ,\quad \mathcal{G}^{\left(
i\right) }=\left( \mathrm{R},{\mathcal{B}},\left\{ Q_{f(t_{i})}:f\in \Sigma
\right\} \right) .
\end{equation*}
It is clear that
\begin{equation*}
\mathcal{F}^{n}=\left( \mathbf{S}^{n},\mathcal{B}(\mathbf{S}^{n}),\left\{
F_{f}^{n}:f\in \Sigma \right\} \right) ,
\end{equation*}
where $F_{f}^{n}=F_{f}^{\left( 1\right) }\times ...\times F_{f}^{\left(
n\right) },$%
\begin{equation*}
F_{f}^{\left( i\right) }=\left\{
\begin{array}{l}
P_{f\left( i/n\right) },\quad \text{if}\quad i\quad \text{is odd,} \\
Q_{f\left( i/n\right) },\quad \text{if}\quad i\quad \text{is even,}
\end{array}
\right.
\end{equation*}
for $i\in J.$ We will show that the global experiments $\mathcal{E}^{n}$ and
$\mathcal{F}^{n}$ are asymptotically equivalent. Toward this end, we note
that by Corollary \ref{CORR-VAR_ST} the local experiments $\mathcal{E}%
_{f}^{\prime \prime ,n}$ and $\mathcal{G}_{f}^{\prime \prime ,n}$ are
asymptotically equivalent uniformly in $f\in \Sigma :$%
\begin{equation}
\sup_{f\in \Sigma ^{\beta }}\Delta \left( \mathcal{E}_{f}^{\prime \prime ,n},%
\mathcal{G}_{f}^{\prime \prime ,n}\right) =o\left( 1\right) .
\label{loc-delta-dist}
\end{equation}
Let $\left\| \cdot \right\| _{\text{Var}}$denote the total variation norm
for measures and let $P_{g}^{\prime \prime ,n},$ $Q_{g}^{\prime \prime ,n}$
be the product measures corresponding to the local experiments $\mathcal{E}%
_{f}^{\prime \prime ,n}$ and $\mathcal{G}_{f}^{\prime \prime ,n}:$%
\begin{equation}
P_{g}^{\prime \prime ,n}=\bigotimes_{i\in J^{\prime \prime
}}P_{g(t_{i})},\quad Q_{g}^{\prime \prime ,n}=\bigotimes_{i\in J^{\prime
\prime }}Q_{g(t_{i})}.  \label{PG.2}
\end{equation}
Then (\ref{loc-delta-dist}) implies that for any $f\in \Sigma $ there is a
Markov kernel $K_{f}^{n}$ such that
\begin{equation}
\sup_{f\in \Sigma }\sup_{h\in \Sigma _{f}(\gamma _{n})}\left\|
K_{f}^{n}\cdot P_{f+h}^{\prime \prime ,n}-Q_{f+h}^{\prime \prime ,n}\right\|
_{\text{Var}}=o\left( 1\right) .  \label{PG.1}
\end{equation}
Let us establish that there is a Markov kernel $M^{n}$ (not depending on $f$%
) such that
\begin{equation}
\sup_{f\in \Sigma }\left\| M^{n}\cdot P_{f}^{n}-F_{f}^{n}\right\| _{\text{Var%
}}=o\left( 1\right) .  \label{PG.2a}
\end{equation}
First note that any vector $x\in X^{n}$ can be represented as $(x^{\prime
},x^{\prime \prime })$ where $x^{\prime }$ and $x^{\prime \prime }$ are the
corresponding vectors in $X^{\prime ,n}$ and $X^{\prime \prime ,n}.$ The
same applies for $s\in \mathbf{S}^{n}:$ $s=(x^{\prime },y^{\prime \prime }),$
where $x^{\prime }\in X^{\prime ,n}$ and $y^{\prime \prime }\in R^{\prime
\prime ,n}.$ For any $x=(x^{\prime },x^{\prime \prime })\in X^{n}$ and $B\in
\mathcal{B}(\mathbf{S}^{n})$ set
\begin{equation*}
M^{n}(x,B)=\int_{R^{\prime \prime ,n}}\mathbf{1}_{B}\left( (x^{\prime
},y^{\prime \prime })\right) K_{\widehat{f}_{n}(x^{\prime })}^{n}(x^{\prime
\prime },dy^{\prime \prime }),
\end{equation*}
where $\widehat{f}_{n}(x^{\prime })$ is the preliminary estimator provided
by Assumption (G1). It is easy to see that
\begin{eqnarray}
\left( M^{n}\cdot P_{f}^{n}\right) (B) &=&\int_{X^{\prime
,n}}\int_{X^{\prime \prime ,n}}M^{n}\left( (x^{\prime },x^{\prime \prime
}),B\right) P_{f}^{\prime \prime ,n}(dx^{\prime \prime })P_{f}^{\prime
,n}(dx^{\prime })  \notag \\
&=&\int_{X^{\prime ,n}}\int_{R^{\prime \prime ,n}}\mathbf{1}_{B}\left(
(x^{\prime },y^{\prime \prime })\right) \left( K_{\widehat{f}_{n}(x^{\prime
})}^{n}\cdot P_{f}^{\prime \prime ,n}\right) (dy^{\prime \prime
})P_{f}^{\prime ,n}(dx^{\prime })  \label{PG.3}
\end{eqnarray}
and
\begin{equation}
F_{f}^{n}(B)=\int_{X^{\prime ,n}}\int_{X^{\prime \prime ,n}}\mathbf{1}%
_{B}\left( (x^{\prime },y^{\prime \prime })\right) Q_{f}^{\prime \prime
,n}(dy^{\prime \prime })P_{f}^{\prime ,n}(dx^{\prime }),  \label{PG.4}
\end{equation}
where $P_{f}^{\prime ,n}$ is the measure in the experiment ${\mathcal{E}}%
^{\prime ,n}$ defined by the analogy with $P_{f}^{\prime \prime ,n}$ in (\ref
{PG.2}), but with $J^{\prime }$ replacing $J^{\prime \prime }.$ By
Assumption (G1),
\begin{equation}
\sup_{f\in \Sigma }P_{f}^{\prime ,n}(A_{f}^{c})=o\left( 1\right) ,
\label{PG.5}
\end{equation}
where $A_{f}=\left\{ x^{\prime }\in X^{\prime ,n}:\left\| \widehat{f}%
_{n}(x^{\prime })-f\right\| _{\infty }\leq c\gamma _{n}\right\} .$ Then (\ref
{PG.3}) and (\ref{PG.4}) imply
\begin{eqnarray*}
&&\left| \left( M^{n}\cdot P_{f}^{n}\right) (B)-F_{f}^{n}(B)\right| \leq
2P_{f}^{\prime ,n}(A_{f}^{c}) \\
&&+\int_{A_{f}}\sup_{f\in \Sigma }\sup_{h\in \Sigma _{f}^{\beta }(\gamma
_{n})}\left\| K_{f}^{n}\cdot P_{f+h}^{\prime \prime ,n}-Q_{f+h}^{\prime
\prime ,n}\right\| _{\text{Var}}P_{f}^{\prime ,n}(dx^{\prime }).
\end{eqnarray*}
Using (\ref{PG.1}) and (\ref{PG.5}) we obtain (\ref{PG.2a}). This implies
that the one-sided deficiency $\delta \left( {\mathcal{E}}^{n},{\mathcal{F}}%
^{n}\right) $ is less that $c_{2}\varepsilon _{n}.$ The bound for $\delta
\left( {\mathcal{F}}^{n},{\mathcal{E}}^{n}\right) $ can be obtained in the
same way; for this we need a result analogous to condition (G1) in the
Gaussian experiment $\mathcal{G}^{n}$. Since the function $\Gamma $ is
smooth and strictly monotone, the existence of a such preliminary estimator
in $\mathcal{G}^{n}$ follows from Korostelev (\cite{koro}). This proves that
the Le Cam distance between $\mathcal{E}^{n}$ and $\mathcal{F}^{n}$ goes to $%
0.$ In the same way we can show that $\mathcal{F}^{n}$ and $\mathcal{G}^{n}$
are asymptotically equivalent. As a result we obtain asymptotic equivalence
of the experiments $\mathcal{E}^{n}$ and $\mathcal{G}^{n}.$ Theorem \ref
{Theorem R-1} is proved.


\section{Appendix}

\subsection{Koml\'os-Major-Tusn\'ady approximation}

Assume that on the probability space $(\Omega ^{\prime },\mathcal{F}^{\prime
},P^{\prime })$ we are given a sequence of independent r.v.'s $%
X_{1},...,X_{n}$ such that for all $i=1,...,n$%
\begin{equation*}
E^{\prime }X_{i}=0,\quad c_{\min }\leq E^{\prime }X_{i}^{2}\leq c_{\max },
\end{equation*}
for some positive absolute constants $c_{\min }$ and $c_{\max }.$ Hereafter $%
\text{E}^{\prime }$ denotes the expectation under the measure $P^{\prime }.$
Assume that the following condition, due to Sakhanenko, is satisfied: there
is a sequence $\lambda _{n},$ $n=1,2,...$ of real numbers, $0<\lambda
_{n}\leq \lambda _{\max },$ where $\lambda _{\max }$ is an absolute
constant, such that for all $i=1,....,n,$%
\begin{equation*}
\lambda _{n}E^{\prime }\left| X_{i}\right| ^{3}\exp \left( \lambda
_{n}\left| X_{i}\right| \right) \leq E^{\prime }X_{i}^{2}.
\end{equation*}
Along with this, assume that on another probability space $(\Omega ,\mathcal{%
F},P)$ we are given a sequence of independent normal r.v.'s $%
N_{1},...,N_{n}, $ such that, for any $i=1,...,n,$%
\begin{equation*}
EN_{i}=0,\quad EN_{i}^{2}=E^{\prime }X_{i}^{2}.
\end{equation*}
Let $\mathcal{H}(1/2,L)$ be a H\"{o}lder ball with exponent $1/2$ on the
unit interval $[0,1],$ i.e. the set of functions satisfying:
\begin{equation*}
\left| f(x)-f(y)\right| \leq L\left| x-y\right| ^{1/2},\quad \left|
f(x)\right| \leq L,\quad x,y\in \lbrack 0,1],
\end{equation*}
where $L$ is an absolute constant.

In the sequel, the notation $\mathcal{L}(X)=\mathcal{L}(Y)$ for r.v.'s means
equality of their distributions.

The following assertion is proved in Grama and Nussbaum \cite{Gr-Nu-1}.

\begin{theorem}
\label{THEOREM-HUNG-CONSTR}A sequence of independent r.v.'s $\widetilde{X}%
_{1},...,\widetilde{X}_{n}$ can be constructed on the probability space $%
(\Omega ,\mathcal{F},P)$ such that $\mathcal{L}(\widetilde{X}_{i})=\mathcal{L%
}(X_{i}),$ $i=1,...,n,$ and
\begin{equation*}
\sup_{f\in \mathcal{H}(1/2,L)}P\left( \left| \sum_{i=1}^{n}f(i/n)\left(
\widetilde{X}_{i}-N_{i}\right) \right| >x(\log n)^{2}\right) \leq c_{0}\exp
\left( -c_{1}\lambda _{n}x\right) ,\quad x\geq 0,
\end{equation*}
where $c_{0}$ and $c_{1}$ are absolute constants.
\end{theorem}

This result is an analog of the functional strong approximation established
by Koltchinskii \cite{Kol} for the empirical processes.

\begin{remark}
Note that the r.v.'s $X_{1},...,X_{n}$ are not assumed to be identically
distributed. We also do not assume any additional richness of the
probability space $(\Omega ,\mathcal{F},P):$ the only assumption is that the
normal r.v.'s $N_{1},...,N_{n}$ exist.
\end{remark}

\subsection{Le Cam and Hellinger distances}

We made use of the following facts.

\textbf{a. }Let $\mathcal{E}=(X,\mathcal{X},\{P_{\theta }:\theta \in \Theta
\})$ and $\mathcal{G}=(Y,\mathcal{Y},\{Q_{\theta }:\theta \in \Theta \})$ be
two experiments with the same parameter space $\Theta .$ Assume that there
is some point $\theta _{0}\in \Theta $ such that $P_{\theta }\ll P_{\theta
_{0}}$ and $Q_{\theta }\ll Q_{\theta _{0}},$ $\theta \in \Theta $ and that
there are versions $\Lambda _{\theta }^{1}$ and $\Lambda _{\theta }^{2}$ of
the likelihoods $dP_{\theta }/dP_{\theta _{0}}$ and $dQ_{\theta }/dQ_{\theta
_{0}}$ on a common probability space $\left( \Omega ,\mathcal{A},P\right) .$
Then the Le Cam deficiency distance between $\mathcal{E}$ and $\mathcal{G}$
satisfies
\begin{equation}
\Delta \left( \mathcal{E},\mathcal{G}\right) \leq \sup_{\theta \in \Theta }%
\frac{1}{2}E_{P}\left| \Lambda _{\theta }^{1}-\Lambda _{\theta }^{2}\right| .
\label{Le Cam-to-L1}
\end{equation}
The proof of this assertion can be found in Le Cam and Yang \cite{LCY-2000},
p. 16 (see also Nussbaum \cite{Nuss}).

Let $H(\cdot ,\cdot )$ denotes the Hellinger distance between probability
measures: if $P$ and $Q$ are probability measures on the measurable space $%
(\Omega ,\mathcal{A})$ and $P\ll \nu $ and $Q\ll \nu ,$ where $\nu $ is a $%
\sigma $-finite measure on $(\Omega ,\mathcal{A}),$ then
\begin{equation}
H^{2}(P,Q)=\frac{1}{2}\int_{\Omega }\left( \left( \frac{dP}{d\nu }\right)
^{1/2}-\left( \frac{dQ}{d\nu }\right) ^{1/2}\right) ^{2}d\nu .
\label{Hell-dist}
\end{equation}
Define the measures $\widetilde{P}_{\theta }$ and $\widetilde{Q}_{\theta }$
by setting $d\widetilde{P}_{\theta }=\Lambda _{\theta }^{1}dP$ and $d%
\widetilde{Q}_{\theta }=\Lambda _{\theta }^{2}dP.$ Using the well-known
relation of the $L_{1}$ norm to $H(\cdot ,\cdot )$ (see Strasser \cite{Str85}%
, 2.15)
\begin{equation}
\frac{1}{2}E_{P}\left| \Lambda _{\theta }^{1}-\Lambda _{\theta }^{2}\right|
\leq \sqrt{2}H(\widetilde{P}_{\theta },\widetilde{Q}_{\theta }).
\label{L1-to-Hell}
\end{equation}

\textbf{b.} Let $P_{1},...,P_{n}$ and $Q_{1},...,Q_{n}$ be probability
measures on $\left( \Omega ,\mathcal{A}\right) .$ Set $P^{n}=P_{1}\times
...\times P_{n}$ and $Q^{n}=Q_{1}\times ...\times Q_{n}.$ Then
\begin{equation}
1-H^{2}\left( P^{n},Q^{n}\right) =\prod_{i=1}^{n}\left( 1-H^{2}\left(
P_{i},Q_{i}\right) \right)  \label{Hell-sum}
\end{equation}
and (cf. Strasser \cite{Str85}, 2.17)
\begin{equation}
H^{2}\left( P^{n},Q^{n}\right) \leq \sum_{i=1}^{n}H^{2}\left(
P_{i},Q_{i}\right) .  \label{Hell-square}
\end{equation}

\textbf{c.} Let $\Phi _{\mu }$ be the normal distribution with mean $\mu $
and variance $1.$ Then
\begin{equation}
H^{2}(\Phi _{\mu _{1}},\Phi _{\mu _{2}})=1-\exp \left( -\frac{1}{8}(\mu
_{1}-\mu _{2})^{2}\right) .  \label{Hell-Gauss}
\end{equation}

\subsection{An exponential inequality}

We made use of the following well-known inequality.

\begin{lemma}
\label{Lemma APX-1}Let $\xi $ be a r.v. such that $E\xi =0$ and $\left| \xi
\right| \leq a,$ for some positive constant $a.$ Then
\begin{equation*}
E\exp \left( \lambda \xi \right) \leq \exp \left( c\lambda ^{2}E\xi
^{2}\right) ,\quad \left| \lambda \right| \leq 1,
\end{equation*}
where $c=e^{a}/2.$
\end{lemma}

\begin{proof}
Set $\mu \left( \lambda \right) =E\exp \left( \lambda \xi \right) $ and $%
\psi \left( \lambda \right) =\log \mu \left( \lambda \right) .$ A simple
Taylor expansion yields
\begin{equation}
\psi \left( \lambda \right) =\psi \left( 0\right) +\lambda \psi ^{\prime
}\left( 0\right) +\frac{\lambda ^2}2\psi ^{\prime \prime }\left( t\right)
,\quad \left| \lambda \right| \leq 1,  \label{APX-1}
\end{equation}
where $\left| t\right| \leq 1,$ $\psi \left( 0\right) =0,$ $\psi ^{\prime
}\left( 0\right) =0,$%
\begin{equation}
\psi ^{\prime \prime }\left( t\right) =\frac{\mu ^{\prime \prime }\left(
t\right) }{\mu \left( t\right) }-\frac{\mu ^{\prime }\left( t\right) ^2}{\mu
\left( t\right) ^2}\leq \mu ^{\prime \prime }\left( t\right) \leq e^aE\xi ^2.
\label{APX-2}
\end{equation}
Inserting (\ref{APX-2}) in (\ref{APX-1}) we get $\psi \left( \lambda
\right) \leq \frac 12e^a\text{E}\xi ^2.$
\end{proof}

\end{document}